\documentclass[final]{siamltex}

\usepackage{amsfonts,amstext,amsmath,amssymb}

\usepackage[OT1]{fontenc}
\usepackage[%
    colorlinks=true,%
    bookmarksnumbered=true,%
    bookmarksopen=true,%
    citecolor=blue,%
    urlcolor=blue,%
    unicode=true,%
    final=true,%
]{hyperref}

\usepackage{subfigure}
\usepackage{rotating}
\usepackage{multirow}
\usepackage{bigstrut}

\newcommand{\ignore}[1]{} 

\newtheorem{remark}{Remark}[section]

\newcommand{\mbb}[1]{\mathbb{#1}} 
\def\One{\mathchoice{\rm 1\mskip-4.2mu l}{\rm 1\mskip-4.2mu l}
{\rm 1\mskip-4.6mu l}{\rm 1\mskip-5.2mu l}}
\newcommand\Ind[1]{{\One_{\{#1\}}}}

\newcommand{\class}{\mbb{C}_{\gamma}}

\newcommand{\Pb}{{\mathsf{P}}}
\newcommand{\Eb}{{\mathsf{E}}}

\newcommand{\Qb}{{\mathsf{Q}}}
\newcommand{\ARL}{{\mathsf{ARL}}}
\newcommand{\ADD}{{\mathsf{ADD}}}
\newcommand{\SADD}{{\mathsf{SADD}}}
\newcommand{\STADD}{{\mathsf{STADD}}}

\newcommand{\mc}[1]{\mathcal{#1}} 
\newcommand{\Jc}{\mc{J}}
\newcommand{\Bc}{{ \mc{B}}}

\newcommand{\xra}{\xrightarrow} 

\newcommand{\abs}[1]{\left\vert#1\right\vert}
\newcommand{\set}[1]{\left\{#1\right\}}
\newcommand{\brc}[1]{\left(#1\right)}
\newcommand{\brcs}[1]{\left[#1\right]}

\renewcommand{\le}{\leqslant} 
\renewcommand{\ge}{\geqslant}

\title{Third-order Asymptotic Optimality of the Generalized Shiryaev--Roberts Changepoint Detection Procedures\thanks{This work was supported by the U.S.\ Army Research Office under MURI grant  W911NF-06-1-0044, by the U.S.\ Air Force Office of Scientific Research under MURI grant FA9550-10-1-0569, by the U.S.\ Defense Threat Reduction Agency under grant HDTRA1-10-1-0086, and by the U.S.\ National Science Foundation under grants CCF-0830419 and EFRI-1025043 at the University of Southern California, Department of Mathematics. The work of Moshe Pollak was also supported by a grant from the Israel Science Foundation and the Marcy Bogen Chair of Statistics at the Hebrew University of Jerusalem, Israel. }}

\author{Alexander\ G.\ Tartakovsky\thanks{Department of Mathematics and
Center for Applied Mathematical Sciences, University of Southern California,
3620 S. Vermont Ave., KAP 108, Los Angeles, CA 90089-2532,
United States of America ({\tt tartakov@usc.edu}).}
        \and Moshe\ Pollak\thanks{Department of Statistics,
Hebrew University of Jerusalem,
Mount Scopus, Jerusalem 91905, Israel ({\tt msmp@mscc.huji.ac.il}).}
        \and Aleksey\ S.\ Polunchenko\thanks{Department of Mathematics and
Center for Applied Mathematical Sciences, University of Southern California,
3620 S. Vermont Ave., KAP 108, Los Angeles, CA 90089-2532,
United States of America ({\tt polunche@usc.edu}).}}

\begin{document}

\maketitle

\begin{center}

\date{February  2011}

\end{center}

\begin{abstract}
Several variations of the Shiryaev--Roberts detection procedure in the context of the simple changepoint problem are considered: starting the procedure at $R_0=0$ (the original Shiryaev--Roberts procedure), at $R_0=r$ for fixed $r>0$, and at $R_0$ that has the quasi-stationary distribution. Comparisons of operating characteristics are made. The differences fade as the average run length to false alarm tends to infinity. It is shown that the Shiryaev--Roberts procedures that start either from a specially designed point $r$ or from the random ``quasi-stationary'' point are third-order asymptotically optimal.
\end{abstract}


%
%
%

\pagestyle{myheadings}
\thispagestyle{plain}
\markboth{A. G. TARTAKOVSKY, M. POLLAK AND A. S. POLUNCHENKO}{OPTIMALITY OF THE GENERALIZED SHIRYAEV-ROBERTS PROCEDURES}

\section{Introduction}\label{s:intro}

The simple changepoint problem posits that one obtains a series of observations $X_1,X_2,\dots$ such that $\{X_i\}_{i\ge 1}$ are independent and, for some value $\nu$, $\nu\ge0$ (the changepoint), $X_1,X_2,\dots,X_{\nu}$ have known density $f$ and $X_{\nu+1}, X_{\nu+2},\dots$ have known density $g$ ($\nu=\infty$ means that all observations have  density $f$ and $\nu=0$ means that all observations have density $g$). The changepoint $\nu$ is unknown, and the sequence $X=\{X_i\}_{i\ge 1}$ is being monitored for detecting a change. A sequential detection policy is defined by a stopping time $T$ (with respect to the $X$'s), so that after observing $X_1,X_2,\dots,X_T$ it is declared that apparently a change is in effect.

By $\Pb_\nu$ we denote the probability measure generated by the observations $X$ when the changepoint is $\nu$ and $\Eb_\nu$  stands for the corresponding expectation. The notation $\nu=\infty$, $\Pb_\infty$ and $\Eb_\infty$ correspond to the no-change scenario. In other words, under $\Pb_\infty$ the observations $\{X_i\}_{i \ge 1}$ are i.i.d. with density $f$ and under $\Pb_0$ the observations $\{X_i\}_{i \ge 1}$ are i.i.d. with density $g$ (both with respect to a dominating measure $\lambda$).

Common operating characteristics of a detection policy $T$ are $\Eb_\infty T$, the average run length (expected time) to false alarm (assuming there is no change), and $\sup_{0\le \nu <\infty} \Eb_\nu(T-\nu|T>\nu)$, the maximal expected delay to detection. Subject to a lower bound $\gamma$ on $\Eb_\infty T$, the goal is to minimize the maximum expected delay. Note that a uniformly optimal procedure that minimizes the expected detection delay $\Eb_\nu(T-\nu|T>\nu)$ for all $\nu\ge0$ over stopping times with $\Eb_\infty T\ge\gamma$ does not exist, and we have to resort to the minimax setting.

In 1961, for the problem of detecting a change in the drift of a Brownian motion, Shiryaev introduced a detection procedure that is a limit of the Bayes procedure when the parameter of the exponential prior distribution tends to zero; see Shiryaev~\cite{Shiryaev:SMD61, Shiryaev:TPA63}. In particular, certain optimality properties of this procedure were established by Shiryaev~\cite{Shiryaev:TPA63}. In discrete time, a similar procedure was first considered by Roberts~\cite{Roberts:T66} as a particular case of the  Girschick-Rubin~\cite{Girschick+Rubin:AMS52} Bayesian procedure by setting the parameter of the geometric prior distribution to zero. Therefore, this procedure is usually referred to as the Shiryaev--Roberts procedure. The Shiryaev--Roberts (SR) procedure and its modifications are the centerpiece of this paper.

Specifically, let $\Lambda_i=g(X_i)/f(X_i)$ denote the likelihood ratio for the observation $X_i$, and define
\begin{align}
R_n&=\sum_{k=1}^n \prod_{i=k}^n \Lambda_i \label{SRstat} ,
\\
T_A&=\inf\{n\ge 1\colon R_n \ge A\}, \label{SRst}
\end{align}
where $A$ is a positive threshold that controls the false alarm rate. For a connection between $A$ and the expected time to false alarm $\Eb_\infty T_A$ see Pollak~\cite{Pollak:AS87}. When defining stopping times we always assume that $\inf\{\varnothing\}=\infty$, i.e., $T_A=\infty$ if $R_n$ never reaches level $A$. Note the recursion
\begin{equation} \label{SRstatrec}
R_{n+1} = (1+R_{n}) \Lambda_{n+1}, \quad n \ge 0, \quad R_0=0
\end{equation}
(with the null initial condition).

Pollak~\cite{Pollak:AS85} tweaked the procedure by starting it off at a random $R_0^{\Qb_A}$ whose distribution is the quasi-stationary distribution $\Qb_A$ of the SR statistic $R_n$, defined by
\begin{equation} \label{QS}
\Qb_A(x) = \lim_{n\to\infty} \Pb_\infty (R_n \le x| T_A> n),
\end{equation}
and showed that the stopping time
\begin{equation} \label{SRPst}
T_A^{\Qb_A} = \inf\{n \ge 1\colon R_n^{\Qb_A} \ge A\},
\end{equation}
where
\begin{equation} \label{SRPstat}
R_{n+1}^{\Qb_A} = (1+R_{n}^{\Qb_A}) \Lambda_{n+1}, \quad n \ge 0, \quad R_0^{\Qb_A}\sim \Qb_A,
\end{equation}
minimizes the maximal expected delay $\sup_{\nu \ge 0} \Eb_\nu (T-\nu|T> \nu)$ asymptotically as $\gamma\to \infty$ to within $o(1)$ over all stopping times that satisfy $\Eb_\infty T \ge \gamma$, where $A$ is such that $\Eb_\infty T_A^{\Qb_A}=\gamma$. We will refer to this randomized SR procedure as the Shiryaev--Roberts--Pollak (SRP) procedure.

Usually, $\Qb_A(x)$ cannot be expressed in a closed form (except in some rare cases). To compute $\Qb_A(x)$ and make the SRP procedure implementable, Moustakides~et~al.~\cite{Moustakidesetal-SS09} proposed a numerical framework.

Until recently the question whether the SRP procedure is exactly optimal (in the class of procedures with $\Eb_\infty T \ge \gamma$) was an open question. Moustakides~et~al.~\cite{Moustakidesetal-SS09} present numerical evidence that there exist procedures that are uniformly better. They regard starting off the original SR procedure at a fixed (but specially designed) $R_0^r=r$, $0 \le r < A$ and defining the stopping time with this new deterministic initialization
\begin{equation}\label{SR-rst}
T_{A}^r = \inf\{n \ge 1\colon R_n^r \ge A\} , \quad A >0,
\end{equation}
where
\begin{equation} \label{SR-rstat}
R_{n+1}^r = (1+R_{n}^r) \Lambda_{n+1}, \quad n \ge 0, \quad R_0^r=r .
\end{equation}
They show by numerical examples that, for certain values of $r$, apparently $\Eb_\nu(T_{A_{*}}^r -\nu | T_{A_{*}}^r > \nu) < \Eb_\nu(T_A^{\Qb_A} -\nu | T_A^{\Qb_A} > \nu)$ for all $\nu \ge 0$, where $A_{*}$ and $A$ are such that $\Eb_\infty T_A^{\Qb_A} = \Eb_\infty T_{A_{*}}^r$ (although the maximal expected delay is only slightly smaller for $T_{A_{*}}^r$). We will refer to the procedure defined in \eqref{SR-rst} and \eqref{SR-rstat} as the SR--$r$ procedure. In~\cite{Moustakidesetal-SS09}, it is conjectured that the SR--$r$ procedure with a specially designed $r=r(\gamma)$ is third-order asymptotically optimal (i.e., to within $o(1)$) in the class of procedures with $\Eb_\infty T \ge \gamma$ as $\gamma \to \infty$. Examples where the SR--$r$ procedure is strictly minimax are provided by Polunchenko~and~Tartakovsky~\cite{Polunchenko+Tartakovsky:AS2010} and Tartakovsky~and~Polunchenko~\cite{Tartakovsky+Polunchenko:IWAP2010}.

Shiryaev~\cite{Shiryaev:SMD61, Shiryaev:TPA63} showed for Brownian motion that if a change takes place after many successive applications (re-runs) of a stopping time $T$ (to a sequence $X_1,X_2,\dots$, starting anew after each false alarm), then the expected delay is minimized asymptotically as $\nu \to\infty$ (i.e., in a stationary mode) over all multi-cyclic procedures with $\Eb_\infty T \ge \gamma$  for every $\gamma >1$ by the original (multi-cyclic) SR procedure. Pollak~and~Tartakovsky~\cite{Pollak+Tartakovsky:SS09} showed the same for discrete time.

The goal of the present paper is to answer questions regarding comparisons between the various SR-type procedures introduced above -- the SR, SR--$r$, and SRP procedures. Is the stationary expected delay of the repeated SR procedure described in the previous paragraph similar to $\lim_{\nu\to\infty} \Eb_\nu(T_A -\nu | T_A >\nu)$? (Yes, see Theorem~\ref{Th2}, Theorem~\ref{Th3} and Corollary~\ref{Cor1}.) What can be said about the maximal expected detection delays of these detection procedures? (The SRP procedure and the SR--$r$ procedure with a specially designed $r$ are third-order asymptotically minimax, i.e., to within a negligible term $o(1) \to 0$. See Theorem~\ref{Th4}. This answer justifies the conjecture of Moustakides~et~al.~\cite{Moustakidesetal-SS09}.) What can be said about $\lim_{\nu\to\infty} \Eb_\nu(T_A -\nu | T_A >\nu)$, $\lim_{\nu\to\infty} \Eb_\nu(T_A^r -\nu | T_A^r >\nu)$, and $\lim_{\nu\to\infty} \Eb_\nu(T_A^{\Qb_A} -\nu | T_A^{\Qb_A} >\nu)$ when all have the same average run length to false alarm $\gamma$? (The average delay to detection at infinity is the smallest for the original SR procedure $T_A$, but the difference between them is $o(1)$ as $\gamma \to \infty$. See Theorems~\ref{Th5} and~\ref{Th4}.) \ We conclude with a numerical example that illustrates these phenomena.

\section{Preliminaries and Heuristics}\label{s:Prelim}

Recall that $\nu$ denotes the unknown changepoint which is identified with the last time instant under the nominal regime.
Thus, conditional on $\nu=k$, the joint density of the vector $(X_1,\dots,X_n)$ can be written as
\begin{equation}\label{pdf}
p(X_1,\dots, X_n|\nu=k)=\prod_{i=1}^{k}f(X_i) \prod_{i=k+1}^n g(X_i)
\end{equation}
for any $n\ge 1$ and $k\ge0$ provided that $\prod_{i=k+1}^n g(X_i)=1$ whenever $k\ge n$.

Given observations $X_1,\dots, X_n$, introduce the hypotheses $H_k: \nu=k < n$ that the change occurs somewhere within this stretch of observations and $H_\infty: \nu = \infty$ that there is no change. Clearly, the latter hypothesis is equivalent to the hypothesis $H_\nu$, $\nu \ge n$. According to \eqref{pdf}, the likelihood ratio of these hypotheses is
\[
\frac{p(X_1,\dots, X_n|H_k)}{p(X_1,\dots, X_n|H_\infty)} = \prod_{i=k+1}^n \Lambda_i,
\]
where $\Lambda_i=g(X_i)/f(X_i)$. Therefore, the SR statistic~\eqref{SRstat} can be interpreted as the average likelihood ratio averaged over a uniform improper prior distribution of the changepoint.

By $T$ we denote a generic stopping time (or a detection procedure) and by
$\class= \{T\colon \Eb_\infty T \ge \gamma\}$  the class of detection procedures
(stopping times) for which the average run length (ARL) to false alarm does not fall below
a given number $\gamma>1$.

The following two objects will be of the main interest in this paper: {\em Supremum Average Delay to Detection}
\[
\SADD(T)=\sup_{0\le\nu<\infty}\Eb_\nu(T-\nu|T>\nu)
\]
and the limiting value of the average detection delay which we will refer to as {\em Average Delay to Detection at Infinity}
\[
\ADD_\infty(T)=\lim_{\nu\to\infty}\Eb_\nu(T-\nu|T>\nu).
\]

As we mentioned in the introduction, we are interested in a minimax setting of minimizing the maximal expected delay $\SADD(T)$ over stopping times with the lower bound on the ARL to false alarm $\Eb_\infty T\ge\gamma$, i.e., in finding a procedure that would minimize $\SADD(T)$ in the class $\class$: $\inf_{T\in\class}\SADD(T) \mapsto T_{opt}$. However, in general we are unable to find an exact solution to this problem for every $\gamma >1$ and, for this reason, we focus on asymptotic solutions for a large ARL to false alarm $\gamma$; see Polunchenko~and~Tartakovsky~\cite{Polunchenko+Tartakovsky:AS2010} and Tartakovsky~and~Polunchenko~\cite{Tartakovsky+Polunchenko:IWAP2010} for examples where an exact minimax solution is available.

\begin{definition}
We call the procedure $T_o\in\class$ {\bf\em first-order asymptotically optimal} if
\[
\lim_{\gamma \to \infty} \frac{\SADD(T_o)}{\inf_{T\in\class}\SADD(T)} =1,
\]
i.e., $\inf_{T\in\class}\SADD(T) = \SADD(T_o)(1+o(1))$  where $o(1)\to0$ as $\gamma \to \infty$.

We call the procedure $T_o\in \class$ {\bf \em second-order asymptotically optimal} if
\[
\inf_{T\in\class}\SADD(T) = \SADD(T_o)+ O(1) \quad \text{as $\gamma\to\infty$},
\]
where $O(1)$ is bounded as $\gamma\to\infty$.

We call the procedure $T_o\in \class$ {\bf\em third-order asymptotically optimal} if
\[
\inf_{T\in\class}\SADD(T) = \SADD(T_o)+ o(1) \quad \text{as $\gamma\to\infty$},
\]
where $o(1)$ tends to zero as $\gamma\to\infty$.
\end{definition}

It follows from Pollak~\cite{Pollak:AS85} that the SRP procedure \eqref{SRPst} is third-order asymptotically optimal whenever $\Eb_0\abs{\log\Lambda_1}<\infty$. In Section~\ref{ss:ADD} we prove the third-order asymptotic optimality property under the stronger second moment condition
$\Eb_0\abs{\log\Lambda_1}^2<\infty$ using  different techniques. The second moment condition allows us to obtain higher-order asymptotic approximations for $\SADD(T_A^{\Qb_A})$ and $\inf_{T\in\class}\SADD(T)$ (up to a vanishing term). Since the SRP procedure is an equalizer, i.e., $\Eb_\nu(T_A^{\Qb_A}-\nu|T_A^{\Qb_A}>\nu)$ does not depend on $\nu$, it is sufficient to evaluate the average run length to detection $\Eb_0 T_A^{\Qb_A}$ assuming that the change is in effect from the very beginning.

More importantly, using the ideas of Moustakides~et~al.~\cite{Moustakidesetal-SS09}, we are able to design the initialization point $r$ in the SR--$r$ procedure \eqref{SR-rst}, which may or may not depend on the false alarm constraint $\gamma$, so that this procedure is also third-order asymptotically optimal. In this respect, the average delay to detection at infinity $\ADD_\infty(T_A^r)$ plays a critical role. To understand why, let us look at Figure~\ref{fig:Fig1} which shows the average detection delay $\Eb_\nu(T_A^r-\nu|T_A^r>\nu)$ versus $\nu$ for several initialization values $r$.  This figure was obtained using integral equations and numerical techniques of Moustakides~et~al.~\cite{Moustakidesetal-SS09}. For $r=0$, this is the classical SR procedure whose average detection delay is monotonically decreasing to its minimum that is attained at infinity (a steady state value). It is seen that there exists a value $r=r^*$ that generally may depend on the threshold $A$ for which the worst point $\nu$ is at infinity, i.e., $\SADD(T_A^{r^*})=\ADD_\infty(T_A^{r^*})$. This is a very important observation, since it allows us to build a proof of asymptotic optimality based on an estimate of $\ADD_\infty(T_A^{r})$. Particular choices of the ``head start" $r^*$ will be discussed in the following sections.

\begin{figure}[!htb]
 \centering
 \includegraphics[width=0.7\textwidth]{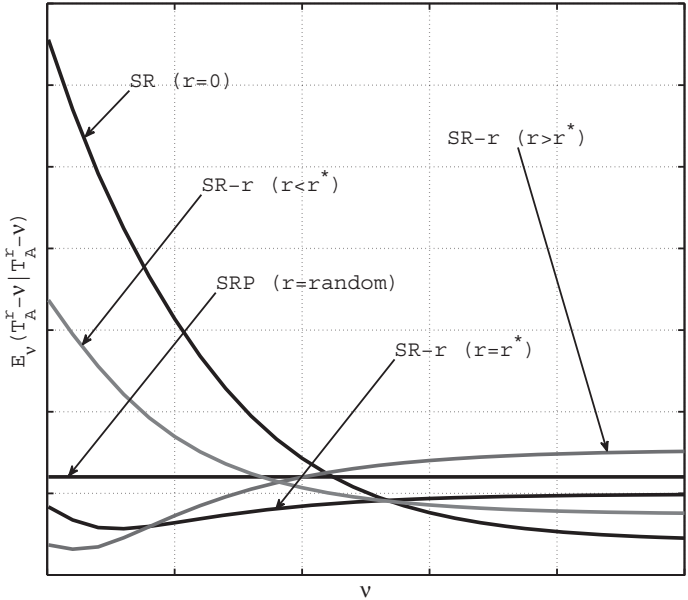}
 \caption{Typical behavior of the expected detection delay as a
 function of changepoint $\nu$ for various initialization strategies.}
 \label{fig:Fig1}
\end{figure}

The monotonicity of the curve for the average detection delay of the SR procedure allows us also to conclude (intuitively only since this is only a numerical observation and there is no theoretical justification of monotonicity) that the asymptotic lower bound for $\inf_{T\in\class}\SADD(T)$ can be evaluated based on the value of $\ADD_\infty(T_A)$. Asymptotically $\Eb_0 T_A^{\Qb_A}$, $\ADD_\infty(T_A^{r_A^{*}})$, and $\ADD_\infty(T_A)$ are the same since the mean of the quasi-stationary distribution is of order $O(\log A)$ and the values of the head start $r^*_A$ that lead to the almost optimal performance are either fixed (i.e., $\lim\limits_{A\to\infty}r_A^{*}= r_\infty^*$)  or go to infinity in such a way that $r^*_A/A \to 0$ as $A\to\infty$, as we will see from the following study.

\section{Asymptotic Performance of the SR--$r$ and SRP Procedures}\label{s:AsPerf}

In this section, we discuss the asymptotic behavior of the SR--$r$ and SRP detection procedures for large values of the  threshold $A$ and the ARL to false alarm $\gamma$.

\subsection{Average Run Length to False Alarm}\label{ss:ARLFA}

Let $Z_i=\log \Lambda_i$ denote the log-likelihood ratio for the $i$-th observation and let $S_n= Z_1+\cdots+Z_n$.
Introduce a one-sided stopping time
\[
\tau_a=\inf\{n\ge 1\colon S_n \ge a\}, \quad a >0.
\]
Let $\kappa_a = S_{\tau_a} - a$ be an overshoot (excess over the level $a$ at stopping), and let
\begin{equation} \label{AveOvershoot}
\zeta =\lim_{a\to\infty} \Eb_0 [ e^{-\kappa_a}], \quad
\varkappa =\lim_{a\to\infty} \Eb_0 \kappa_a .
\end{equation}
The constants $\zeta$ and $\varkappa$ depend on the model and can be computed numerically. In general, $0<\zeta<1$ and $\varkappa>0$.

\begin{theorem} \label{Th1}
Assume that $r=r^*$ where $r^*$ is either fixed or, more generally, $r^*=r_A^* \to \infty$ in such a way that $r_A^*/A \to 0$ as $A\to\infty$. Then for the SR--$r$ procedure, uniformly in $0 \le r \le r_A^*$,
\begin{equation}\label{ARLrA}
\Eb_\infty T_A^r = (A/\zeta) (1+o(1)) \quad \text{as $A \to \infty$},
\end{equation}
where the constant $\zeta$ is defined in~\eqref{AveOvershoot}.

For the SRP procedure
\begin{equation}\label{ARLSRP}
\Eb_\infty T_A^{\Qb_A}=(A/\zeta)(1+o(1)) \quad \text{as $A\to\infty$}.
\end{equation}
\end{theorem}

The proof of this theorem is given in the Appendix.

Let $R_\infty$ denote a random variable that asymptotically as $n\to\infty$ has the same $\Pb_\infty$-distribution as $R_n^r$, i.e.,
\[
\Pb(R_\infty \le x) = \lim_{n\to\infty}\Pb_\infty(R_n^r \le x) := \Qb_{\mathrm{st}}(x),
\]
where $\Qb_{\mathrm{st}}(x)$ is called the stationary distribution of $R_n^r$. Recall also that  we denote by $\Qb_A(x) = \lim_{n\to\infty} \Pb_\infty (R_n^r \le x| T_A^r> n)$
the quasi-stationary distribution of $R_n^r$.

We always assume that both quasi-stationary $\Qb_A(x)$ and stationary $\Qb_{\mathrm{st}}(x)$ distributions exist, which is always true when $\Lambda_1$ is continuous.

Note that the process $\{R_n^r-n-r\}_{n \ge 0}$ is a zero mean $\Pb_\infty$-martingale
and, hence, applying the  optional sampling theorem yields $\Eb_\infty T_A^r =
\Eb_\infty R_{T_A^r} -r$, which can be used to approximate $\Eb_\infty T_A^r$.
Using the above fact along with Theorem~\ref{Th1}, for practical purposes we suggest the following approximations
\begin{equation}\label{ARLapproxSR}
\Eb_\infty T_A^{r} \approx A/ \zeta -r
\end{equation}
and
\begin{equation}\label{ARLapproxSRP}
\Eb_\infty T_A^{\Qb_A} \approx A/ \zeta -\mu_A ,
\end{equation}
where $\mu_A=\int_0^A x\,d\Qb_A(x)$ is the mean of the quasi-stationary distribution.

Note that the mean of the quasi-stationary distribution is of order $O(\log A)$ as $A\to\infty$. Indeed, by Kesten~\cite[Theorem~5]{Kesten:AM73},
\begin{equation}\label{Stas}
\lim_{n\to\infty} \Pb_\infty (R_n > x)= 1- \Qb_{\mathrm{st}}(x) \sim 1/x \quad \text{as $x \to \infty$},
\end{equation}
which along with the fact that $\Qb_A(x) \ge \Qb_{\mathrm{st}}(x)$ (cf.~Pollak~and~Siegmund~\cite{Pollak+Siegmund:JAP86}) yields
\begin{equation}\label{meanQSupper}
\mu_A = \int_0^A [1-\Qb_A(x)]\,dx \le \int_0^A [1-\Qb_{\mathrm{st}}(x)]\,dx = O(\log A).
\end{equation}
Since $\Qb_A(x) \to \Qb_{\mathrm{st}}(x)$ as $A\to\infty$, it follows that $\mu_A= \log A +C_A$, where $C_A=O(1)$ as $A\to\infty$, so that
\begin{equation}\label{meanQS}
\mu_A = \log A + O(1) \quad \text{as}~~ A \to \infty.
\end{equation}

\subsection{Average Delay to Detection and Asymptotic Optimality} \label{ss:ADD}

We continue with obtaining asymptotic approximations (as $A\to\infty$) for the average delay to detection
$\Eb_\nu( T_A^r-\nu | T_A^r >\nu)$, including the case of the large changepoint $\nu$, i.e., for $\ADD_\infty(T_A^r)$,
as well as with deriving an asymptotic lower bound for $\inf_{T\in\class}\SADD(T)$. This will allow us to ascertain whether the SR--$r$ procedure with a certain initialization $r$ (which is either fixed or may depend on $\gamma$) is third-order asymptotically optimal as $\gamma\to\infty$.

Recall that $Z_i=\log\Lambda_i$ is the log-likelihood ratio for the observation $X_i$, $S_n=\sum_{i=1}^n Z_i$ and  $\varkappa$ is the limiting average overshoot in the one-sided test $\tau_a=\min\{n\ge1\colon S_n \ge a\}$ defined in~\eqref{AveOvershoot}. Let $S_n^j = \sum_{i=j}^n Z_i$ and let $V_{\nu,\infty} = \sum_{n=\nu+1}^{\infty} e^{-S_n^{\nu+1}} $. Let
\[
I=\Eb_0 Z_1 = \int \log \brc{\frac{g(x)}{f(x)}} g(x) \lambda(dx)
\]
denote the Kullback--Leibler information number. 

\begin{lemma}\label{Lem1} 
Let $\Eb_0 |Z_1|^2<\infty$ and assume that $Z_1$ is non-arithmetic. Let $0<N_A<A$ be such that  $N_A/(A^{1-\delta} \log A) \to \infty$ and $N_A=o(A/\log A)$ as $A \to \infty$ for some $\delta\in (0,1)$. Let $r \ge 0$ and let $T_A^r$ be defined as in~\eqref{SR-rst}. Then, as $A\to\infty$,
\begin{equation}\label{AAAnu}
\begin{aligned}
\Eb_\nu (T_A^r-\nu | T_A^r >\nu, R_\nu^r)&= \frac{1}{I}\Bigg\{\log A + \varkappa - \log(1+R_\nu^r)
\\
 & \quad - \Eb_\nu\brcs{\log\brc{1+ \frac{V_{\nu, \infty}}{1+R_\nu^r}} \Bigg | T_A^r >\nu, R_\nu^r} \Bigg\} + o(1) ,
\end{aligned}
\end{equation}
where $o(1) \to 0$ as $A\to \infty$ uniformly on $\{N_A \le \nu < \infty, R_\nu^{r}< A/N_A, 0 \le r < \infty\}$.
\end{lemma}

The proof of Lemma~\ref{Lem1} is given in the Appendix.

\begin{remark} \label{Rem3.1} 
Let $V_\infty = \sum_{j=1}^{\infty} e^{-S_j} $. Note that $V_{\nu, \infty}$ is independent of $R_\nu^r$ and has the same
$\Pb_\nu$-distribution of all $\nu \ge 1$, i.e., it is distributed as  $V_\infty$ under $\Pb_0$.
\end{remark}

Recall that by $R_\infty$ we denote a random variable that has the $\Pb_\infty$-limiting (stationary) distribution of $R_n$ as $n \to\infty$, i.e.,  $\Qb_{\mathrm{st}}(x) = \lim_{n\to\infty} \Pb_\infty(R_n \le x) = \Pb( R_\infty \le x)$. Let
\begin{equation}\label{constant1}
C_\infty=\Eb[\log (1+R_\infty +V_\infty)] = \int_0^\infty \int_0^\infty \log(1+x+y)\,d\Qb_{\mathrm{st}}(x) \,d\widetilde{\Qb}(y),
\end{equation}
and
\begin{equation}\label{constant2}
C_r=\Eb[\log (1+r +V_\infty)] = \int_0^\infty \log(1+r+y)\,d\widetilde{\Qb}(y),
\end{equation}
where $\widetilde{\Qb}(y)=\Pb_0 (V_\infty \le y)$.

The following theorem, whose proof is based on Lemma~\ref{Lem1} and can be found in the Appendix,  provides asymptotic approximations (for large $A$) for the average delay to detection of the SR--$r$ procedure (for large $\nu$ and $\nu=0$), and for the supremum average delay to detection $\SADD(T_A^{\Qb_A})=\Eb_0 T_A^{\Qb_A}$ of the SRP procedure (within vanishing terms $o(1)$).

\begin{theorem}\label{Th2} 
If $\Eb_0 |Z_1|^2<\infty$ and $Z_1$ is non-arithmetic, then for any $r \ge 0$
\begin{equation}\label{ADDinfty}
 \ADD_\infty(T_A^r) =\Eb_0 T_A^{\Qb_A}= \frac{1}{I} \brcs{\log A + \varkappa - C_\infty} +o(1)  \quad \text{as $A\to \infty$},
\end{equation}
and
\begin{equation}\label{ADDzeroSRr}
 \Eb_0 T_A^r = \frac{1}{I} \brcs{\log A + \varkappa - C_r} +o(1)  \quad \text{as $A\to \infty$},
\end{equation}
where $o(1) \to 0$ as $A\to\infty$.
\end{theorem}

Define
\begin{equation} \label{J}
\Jc(T) = \frac{\sum_{\nu=0}^\infty \Eb_\nu(T-\nu|T>\nu)\Pb_\infty(T>\nu)}{\Eb_\infty T}.
\end{equation}

The following lemma provides the lower bound for the supremum average delay to detection in the class $\class$. This bound will be used to obtain an asymptotic lower bound in Theorem~\ref{Th3}  and for the proof of third-order asymptotic optimality of detection procedures in Theorem~\ref{Th4}.

\begin{lemma} \label{Lem2} 
Let $T_A$ be the stopping time of the SR procedure that starts from zero and let the threshold $A=A_\gamma$ be chosen so that $\Eb_\infty T_A=\gamma$. The following lower bound holds:
\begin{equation}\label{LBnonas}
\inf_{T\in \class} \SADD(T) \ge \Jc(T_{A_\gamma}).
\end{equation}
\end{lemma}

\begin{proof}
Obviously, for any stopping time $T$
\[
\begin{aligned}
\SADD(T)&=\frac{\sum_{k=0}^\infty [\sup_\nu \Eb_\nu(T-\nu|T>\nu)] \Pb_\infty(T> k)}{\Eb_\infty T}\\
&\ge\frac{\sum_{k=0}^\infty \Eb_k(T-k|T>k) \Pb_\infty(T>k)}{\Eb_\infty T} = \Jc(T).
\end{aligned}
\]
As follows from Pollak~and~Tartakovsky~\cite{Pollak+Tartakovsky:SS09}, the right hand side is minimized by the SR stopping time $T_{A_\gamma}$, so that
\[
\inf_{T\in\class} \SADD(T) \ge \inf_{T\in\class} \Jc(T) = \Jc(T_{A_\gamma})
\]
and the proof is complete.
\end{proof}

The following theorem provides the asymptotic approximation for the lower bound $\Jc(T_A)$. Its proof is given in the Appendix.

\begin{theorem} \label{Th3} 
Let $\Jc(T)$ be defined as in \eqref{J} and $C_\infty$ as in~\eqref{constant1}. If $\Eb_0 |Z_1|^2<\infty$ and $Z_1$ is non-arithmetic, then
\[
\Jc(T_A)=\frac{1}{I} (\log A + \varkappa - C_\infty) +o(1)  \quad \text{as $A\to \infty$},
\]
where $o(1) \to 0$ as $A\to\infty$.
\end{theorem}

Theorem~\ref{Th3} also allows for the following interpretation. Consider the following multi-cyclic detection procedure. Let a stopping time $T$ be applied repeatedly after each alarm, so that  $T_1, T_2, \dots$ are independent copies of $T$ and $T_j$ is the time interval between the $(j-1)$th and $j$th alarms. Clearly, the number $\ell_\nu$ of false alarms before the changepoint $\nu$ is
\begin{equation} \label{ell}
\ell_\nu=\max\{i\colon T_1+\cdots+T_i \le \nu\} ,
\end{equation}
and the real change occurring at the point $\nu+1$  is detected at the time $\mathcal{T}_{\ell_\nu}= T_1+\cdots+T_{\ell_\nu+1}$.
For any fixed $\nu \ge 0$, the average delay to detection of the multi-cyclic (repeated) detection procedure is $\Eb_\nu(\mathcal{T}_{\ell_\nu}-\nu)$. Assuming that the change occurs at a far time horizon (i.e., $\nu\to\infty$), introduce the {\em stationary average detection delay}
\[
\STADD(T)=\lim_{\nu \to\infty} \Eb_\nu(\mathcal{T}_{\ell_\nu} -\nu) .
\]
By applying renewal theory, it can be shown that $\STADD(T) = \Jc(T)$. To this end, see Pollak~and~Tartakovsky~\cite{Pollak+Tartakovsky:SS09}. Furthermore, the SR procedure is exactly optimal in the sense of minimizing the STADD.

Therefore, the following corollary is a direct consequence of Theorem~\ref{Th2} and Theorem~\ref{Th3}.

\begin{corollary} \label{Cor1}
If $Z_1$ is non-arithmetic and $\Eb_0 |Z_1|^2<\infty$, then, as $A\to\infty$,
\[
\ADD_\infty(T_A) = \STADD(T_A) + o(1)
\]
and
\[
\STADD(T_A)=\frac{1}{I} (\log A + \varkappa - C_\infty) +o(1) .
\]
\end{corollary}

The following theorem establishes asymptotic optimality of the SRP and SR--$r$ detection procedures under moderate conditions. Its proof is immediate from the above results.

\begin{theorem}\label{Th4} 
Let $\Eb_0\abs{Z_1}^2<\infty$ and let $Z_1$ be non-arithmetic.

\noindent{\rm (i)} Then
\begin{equation}\label{LBasym}
\inf_{T\in\class}\SADD(T) \ge \frac{1}{I} \brcs{\log (\gamma \zeta) + \varkappa - C_\infty} +o(1)  \quad \text{as $\gamma\to \infty$}.
\end{equation}

\noindent{\rm (ii)} If in the SRP procedure $A=A_\gamma =\gamma\zeta$, then $\Eb_\infty T_A^{\Qb_A} = \gamma (1+o(1))$ and
\begin{equation}\label{SADDSRP}
 \SADD(T_{A}^{\Qb_A}) = \frac{1}{I} \brcs{\log (\gamma \zeta) + \varkappa - C_\infty} +o(1)  \quad \text{as $\gamma\to \infty$}.
\end{equation}
Therefore, the SRP procedure is asymptotically third-order optimal in the class $\class${\rm :}
\[
\inf_{T\in \class} \SADD(T) = \SADD(T_{A}^{\Qb_A}) + o(1) \quad \text{as $\gamma\to\infty$}.
\]

\noindent{\rm (iii)} If in the SR--$r$ procedure $A=A_\gamma=\gamma \zeta$, and the initialization point $r$ is either fixed or tends to infinity with the rate $o(\gamma)$ and is selected so that $\SADD(T_A^r) =  \ADD_\infty(T_{A}^r)$, then $\Eb_\infty T_A^r = \gamma (1+o(1))$ and
\begin{equation}\label{SADDSR-r}
 \SADD(T_{A}^r) = \frac{1}{I} \brcs{\log (\gamma \zeta) + \varkappa - C_\infty} +o(1)  \quad \text{as $\gamma\to \infty$}.
\end{equation}
Therefore, the SR--$r$ procedure is asymptotically third-order optimal{\rm :}
\[
\inf_{T\in \class} \SADD(T) = \SADD(T_{A}^r) + o(1) \quad \text{as $\gamma\to\infty$}.
\]
\end{theorem}

\begin{proof}
The asymptotic lower bound \eqref{LBasym} follows from Lemma~\ref{Lem2} and Theorem~\ref{Th3}.
The asymptotic approximations \eqref{SADDSRP} and \eqref{SADDSR-r} (and statements of (ii) and (iii) in whole) follow from Theorems~\ref{Th1} and \ref{Th2}.
\end{proof}

\begin{remark}\label{Rem3.2}
Third-order asymptotic optimality of the SRP procedure follows from Pollak~\cite{Pollak:AS85} (under the sole first moment condition),
so that this result is not new. However, higher-order asymptotic approximation for the SADD~\eqref{SADDSRP} is new.
\end{remark}

\begin{remark}\label{Rem3.3} 
Feasibility of selecting $r$ so that $\SADD(T_A^r)=\ADD_\infty(T_{A}^r)$ follows from numerical experiments performed by Moustakides~et~al.~\cite{Moustakidesetal-SS09} as well as from the example in Section~\ref{s:Example}. See Figure~\ref{fig:Fig1} in Section~\ref{s:Prelim} and Figures~\ref{fig:OC_vs_k} and \ref{fig:SRP_SR_r_ADDk_vs_k_ARL_100} in Section~\ref{s:Example}.
\end{remark}

\begin{remark}\label{Rem3.4} 
Since for the SR procedure $\SADD(T_A)= \Eb_0 T_A$, it follows from Theorem~\ref{Th2} (setting $r=0$ in \eqref{ADDzeroSRr}) that
\[
\SADD(T_A) = \frac{1}{I} (\log A +\varkappa-C_0) + o(1) \quad \text{as $A\to \infty$},
\]
where $C_0 = \Eb_0[\log(1+V_\infty)]$. Since $A=\zeta \gamma$ implies $\Eb_\infty T_A=\gamma(1+o(1))$, it follows that with this choice of threshold
\begin{equation}\label{SADDSR}
\SADD(T_A)= \frac{1}{I} [\log(\zeta \gamma) + \varkappa - C_0] + o(1) \quad \text{as $\gamma \to \infty$}.
\end{equation}
Comparing \eqref{SADDSR} with the lower bound \eqref{LBasym} shows that
\[
\inf_{T\in \class} \SADD(T) = \SADD(T_A) + O(1) \quad \text{as $\gamma \to \infty$} .
\]
Thus, the SR procedure is only second-order asymptotically optimal and the difference is approximately equal to $(C_\infty -C_0)/I$. This difference can be quite large when detecting small changes {\rm(}i.e., when $I$ is small{\rm)}.
\end{remark}

It is worth noting that Theorem~\ref{Th2} suggests that if the initializing point $r=r^*$ is selected from the equation $C_r = C_\infty$, then for the large ARL to false alarm  $\gamma$ the values of the average delays to detection at zero and infinity are approximately equal, $\Eb_0 [T_A^{r^*}] \approx \ADD_\infty(T_A^{r^*})$ (to within small terms $o(1)$). This choice of the head-start is intuitively appealing since we intend to make the SR$-r$ procedure look like an equalizer as much as possible. Obviously, the value of $r^*$ does not depend on $\gamma$, i.e., it is a fixed number that depends on the model. This observation will be further elaborated in Section~\ref{s:Example}. Note also that the fact that the limiting value $\lim_{A\to\infty} r^*_A = r^*$ equating $\Eb_0 [T_A^{r^*}]$ and $\ADD_\infty(T_A^{r^*})$ to within $o(1)$ is a fixed number has been first noticed by Moustakides and Tartakovsky~\cite{MoustTartakovUM} for the problem of detecting a change in the drift of a Brownian motion.  Also,  although starting at $r^*$ causes for a faster initial response than starting at $r=0$, the resemblance to Lucas and Crosier's \cite{Lucas&Crosier82} FIR scheme is secondary:  their method is designed to give a really fast initial response, whereas our goal is to attain asymptotic third-order optimality.   

It is interesting to ask how the average detection delays at infinity $\ADD_\infty(T_A)$, $\ADD_\infty(T_A^r)$, and $\ADD_\infty(T_A^{\Qb_A})$ are related when all three procedures have the same average run length to false alarm $\gamma$. It turns out that the $\ADD_\infty$ is the smallest for the original SR procedure $T_A$. Theorem~\ref{Th5} below proves this statement. Note also that by Theorems~\ref{Th2} and~\ref{Th4} the difference between ADD's of all three procedures  is $o(1)$ as $\gamma\to\infty$.

This result can be proven in two steps: 1) To show that the ARL to false alarm $\Eb_\infty T_A^{\Qb_A}$ of the SRP procedure is increasing in $A$, the threshold (the fact that the ARL to false alarm of the SR--$r$ procedure is increasing in $A$ for a fixed $r$ is obvious); and 2) To show that the average delay $\ADD_\infty(T_A^{\Qb_A})$ of the SRP procedure is increasing in $A$ (obviously, the ADD's at infinity are the same for all three procedures when the same threshold is being used). Since the SR procedure requires the lowest threshold to attain the same false alarm rate, this implies that the SR procedure has the lowest $\ADD_\infty$. We believe that  $\Eb_\infty T_A^{\Qb_A}$ and $\ADD_\infty(T_A^{\Qb_A})$ are both increasing in $A$ in the general case. However, we are able to prove this fact only when the cumulative distribution function of $\log \Lambda_1$ is concave, both pre-change and post-change, something that guarantees  monotonicity properties of the Markov detection statistics.  This is restrictive, but it does hold, for example, in detection of a shift of a normal mean and in detection of a change of the parameter of an exponential distribution.  It also holds for the example considered in the next section.

For $\eta >0$, regard the sequence defined by the recursion
\begin{equation}\label{Reta}
R_{n+1}^{(\eta)} = \brc{\eta + R_n^{(\eta)}} \Lambda_{n+1}, \quad R_0^{(\eta)} =r.
\end{equation}

To prove the required result we need the following lemma whose proof can be found in the Appendix.

\begin{lemma}\label{Lem3} 
Let $F$ be a cumulative distribution function of $\log \Lambda_1$ that is log-concave {\rm(}i.e., $\log F(x)$ is a concave function{\rm)}. Then the process $(M_n)_{n\ge 0}$ that has transition probabilities
\[
\Pb (M_{n+1} \le x | M_n =t) = \Pb \brc{ R_{n+1}^{(\eta)} \le x | R_n^{(\eta)} = t, R_{n+1}^{(\eta)} < A}
\]
is a stochastically monotone Markov process, i.e., $\Pb (M_{n+1} > x | M_n =t)$ is non-decreasing and right-continuous in $t$ for all $x$.
\end{lemma}

\begin{remark}\label{Rem3.5} 
Note that the normal cumulative distribution function $ \Phi(\frac{x-\mu}{\sigma}) $ is log-concave, so the log-likelihood ratio of two normals whose means differ has a log-concave cdf.  The same applies to two differing exponential distributions as well as to two differing beta-distributions considered in Section~\ref{s:Example}.
\end{remark}

We are now prepared to state the desired result. The details of the proof are given in the Appendix.

\begin{theorem}\label{Th5} 
Suppose that the cdf of $\log \Lambda_1$ is log-concave both pre-change and post-change.
Let $1<\gamma<\infty$ be fixed, and let $A_\gamma^r$ be such that the ARL to false alarm of the SR--$r$ procedure $T_{A_\gamma^r}^r=\inf\{n\ge1\colon R_n^r\ge A_\gamma^r\}$ is $\gamma$.
Then $\ADD_\infty(T_{A_\gamma^r}^r)$ is an increasing function of $r$ and
\[
\min_{0 \le r < \infty} \ADD_\infty(T^r_{A_\gamma^r})  = \ADD_\infty(T^0_{A_\gamma^0}) <
\ADD_\infty(T^{\Qb_A}_{A_\Qb}),
\]
where $A_{\Qb}$ is such that  $\Eb_\infty T^{\Qb_A}_{A_\Qb}= \gamma$.
\end{theorem}

\subsection{Computing Constants $C_r$ and $C_\infty$}\label{ss:Constants}

In order to implement the asymptotic approximations  we have to be able to compute the constants $C_r$ and
$C_\infty$ defined by
\begin{equation}
C_r=\int_0^\infty\log (1+r+y)\,d\widetilde{\Qb}(y),\quad C_\infty=\int_0^\infty\int_0^\infty \log (1+y +x)\,d\Qb_{\mathrm{st}}(x)\,d\widetilde{\Qb}(y),
\label{asympC}
\end{equation}
where $\Qb_{\mathrm{st}}(x)=\lim_{n\to\infty}\Pb_\infty(R_{n}\le x)$ is the stationary distribution of the SR statistic $R_{n}$ under $\Pb_\infty$ and $\widetilde{\Qb}(y)=\lim_{n\to \infty}\Pb_0(V_{n}\le y)$ is the limiting distribution of $V_n=\sum_{i=1}^n e^{-S_i}$ under $\Pb_0$.

Assume that the distribution of $\Lambda_1$ is continuous. Then for computing the constants we need to evaluate the two densities $q_{\mathrm{st}}(x)=d\Qb_{\mathrm{st}}(x)/dx$ and $\tilde{q}(x)=d\widetilde{\Qb}(x)/dx$. Let $R_\infty$ and $V_\infty$ be random variables that are the limit (in distribution, as $n \to \infty$) of $R_n$ and $V_n$, respectively, which have densities $q_{\mathrm{st}}(x)$ and $\tilde{q}(x)$. To find the desired densities, observe that, by recursion~\eqref{SRstatrec}, $R_\infty$ and $(1+R_\infty)\Lambda_1$ have the same density $q_{\mathrm{st}}(x)$ under $\Pb_\infty$. Similarly $V_\infty$ and $(1+V_\infty)\Lambda_1^{-1}$ have the same density $\tilde{q}(x)$ under $\Pb_0$. To see this note that, by the i.i.d.\ property of the data, $V_n$ has the same $\Pb_0$-distribution as the random variable $\widetilde{V}_n=\sum_{i=1}^n \prod_{j=i}^n \Lambda_j^{-1}$, which follows the recursion $\widetilde{V}_n=(1+\widetilde{V}_{n-1})\Lambda_n^{-1}$. Therefore, we have the following integral equations for these densities
\[
q_{\mathrm{st}}(x)=\int_0^\infty q_{\mathrm{st}}(y)\left[\dfrac{\partial}{\partial x}F_\infty\left(\frac{x}{1+y}\right)\right]dy;~~
\tilde{q}(x)=-\int_0^\infty \tilde{q}(y)\left[\dfrac{\partial}{\partial x}F_0\left(\frac{1+y}{x}\right)\right]dy,
\]
where $F_\infty(x) = \Pb_\infty(\Lambda_1 \le x)$ and $F_0(x) = \Pb_0(\Lambda_1 \le x)$ are the corresponding distribution functions of the likelihood ratio $\Lambda_1$.

Thus, $q_{\mathrm{st}}(x)$ and $\tilde{q}(x)$ are the eigenfunctions corresponding to the unit eigenvalues of the linear operators defined, respectively, with the kernels
\[
\mathcal{K}_\infty(x,y)=\dfrac{\partial}{\partial x}F_\infty\left(\frac{x}{1+y}\right),~~
\mathcal{K}_0(x,y)=-\dfrac{\partial}{\partial x}F_0\left(\frac{1+y}{x}\right).
\]
The constants $C_r$ and $C_\infty$ are then obtained, usually by numerical integration.

The next section offers a comparative performance analysis for an example where $C_r$ and $C_\infty$ are computable analytically.

\section{Accuracy of Asymptotic Approximations: An Example}
\label{s:Example}

To verify the accuracy of the asymptotic approximations, we carried out an extensive performance evaluation of the procedures discussed in the earlier sections for the following example. Suppose $\{X_n\}_{n\ge1}$ is a series of independent observations such that $X_1,X_2,\ldots,X_{\nu}$ are $\mathsf{beta}(2,1)$ each, and $X_{\nu+1},X_{\nu+2},\ldots$ are $\mathsf{beta}(1,2)$ each. Put another way, the series undergoes a sudden and abrupt shift in the expected value from $2/3$ pre-change to $1/3$ post-change, while retaining the variance.

To be specific, our goal is to verify the conditions and the accuracy of the asymptotic approximations given in Theorem~\ref{Th2}, Theorem~\ref{Th4}, and Remark~\ref{Rem3.4}, i.e.,
\begin{equation}\label{SADDapprox}
\begin{aligned}
\SADD(T_A^r) & \approx \SADD(T_A^{\Qb_A})\approx \frac{1}{I}\brcs{\log A+\varkappa-C_\infty} ,
\\
\SADD(T_A) & \approx   \frac{1}{I}\brcs{\log A+\varkappa-C_0},
\end{aligned}
\end{equation}
and the approximations for the ARL to false alarm given in~\eqref{ARLapproxSR}, i.e.,
\begin{equation}\label{ARLapprox}
\Eb_\infty T_A^{r}\approx A/\zeta -r \quad\text{and}\quad \Eb_\infty T_A^{\Qb_A} \approx A/ \zeta-\mu_A,
\end{equation}
where $\mu_A$ is the mean of the quasi-stationary distribution.

To undertake this task, it is necessary to be able to calculate the constants $C_r$, $C_\infty$, $\zeta$, and $\varkappa$ and also to compute the initialization point $r$ and the mean $\mu_A$ of the quasi-stationary distribution $\Qb_A$. While usually the constants $C_r$ and $C_\infty$ can be evaluated only numerically or by Monte Carlo, it turns out that for the $\mathsf{beta}$-model considered these constants are computable analytically.

The pre- and post-change probability densities for this scenario are
\begin{align*}
f(x)=2x\Ind{0\le x\le1}\quad\text{and}\quad g(x)=2(1-x)\Ind{0\le x\le1},
\end{align*}
respectively, and the likelihood ratio for the $n$th observation is $\Lambda_n=1/X_n-1$.

The quasi-stationary distribution satisfies the integral equation
\[
\Qb_A(x)=\frac{1}{\lambda_{A}}\int_0^A  F_\infty\left(\frac{x}{1+y}\right) d\Qb_A(y),
\]
where $F_\infty(t) = \Pb_\infty(\Lambda_1\le t)$ and
\[
\lambda_A =  \int_0^A  F_\infty\left(\frac{A}{1+y}\right) d\Qb_A(y).
\]
Since for any $t\ge0$
\begin{equation}\label{Finfty}
\Pb_\infty(\Lambda_1\le t) =1-\Pb_\infty\left(X_n\le\frac{1}{1+t}\right)=1-(1+t)^{-2},
\end{equation}
which is continuously differentiable with respect to $t$, the equation for the quasi-stationary distribution $\Qb_A(x)$ for this model is
\begin{equation}\label{InteqQ}
\lambda_A \, \Qb_A(x)=1-\int_0^A\frac{(1+y)^2}{(1+x+y)^2}\,d\Qb_A(y).
\end{equation}

Note that the quasi-stationary distribution converges to the stationary distribution $\Qb_A(x) \to \Qb_{\mathrm{st}}(x)$ as $A\to\infty$ (cf.\ Pollak and Siegmund~\cite{Pollak+Siegmund:JAP86}).

Equation \eqref{InteqQ} cannot be solved analytically for an arbitrary {\em finite} $A$, but its limiting value when $A\to\infty$ (i.e., the stationary distribution of the statistic $R_n^r$) does permit  a closed-form solution. By~\eqref{InteqQ}, the stationary distribution $\Qb_{\mathrm{st}}(x)$ satisfies the following equation
\begin{align*}
\Qb_{\mathrm{st}}(x)&=1-\int_0^\infty\frac{(1+y)^2}{(1+x+y)^2}\,d\Qb_{\mathrm{st}}(y),
\end{align*}
and the solution is $\Qb_{\mathrm{st}}(x)=[x/(1+x)]\Ind{x\ge0}$.

To derive the equation for $\widetilde{\Qb}(x)$, observe first that for $t\ge0$
\begin{align*}
\Pb_0(1/\Lambda_1\le t)&=\Pb_0(\Lambda_1\ge1/t)=\Pb_0\left(X_n\le\frac{t}{1+t}\right)\\
&=\frac{t}{1+t}\left(2-\frac{t}{1+t}\right)=1-(1+t)^{-2},
\end{align*}
which is identical to $\Pb_\infty(\Lambda_1\le t)$. As a result, the distribution $\widetilde{\Qb}(x)$ satisfies precisely the same equation as $\Qb_{\mathrm{st}}(x)$ and, therefore,
\begin{equation} \label{Qstzerobeta}
\widetilde{\Qb}(x)=\Qb_{\mathrm{st}}(x) =\frac{x}{1+x} \Ind{x\ge0}.
\end{equation}

Using \eqref{asympC} and \eqref{Qstzerobeta}, one is able to calculate the constants $C_r$ and $C_\infty$ {\em exactly} as
\begin{equation}\label{C0infty}
C_r= \frac{1+r}{r} \log (1+r) \quad\text{and}\quad C_\infty=\frac{\pi^2}{6}\approx 1.6449.
\end{equation}
In particular, $C_0=1$.

Note that the Kullback-Leibler information number $I=1$, so that
\[
\SADD(T_A^r) \approx \SADD(T_A^{\Qb_A})\approx\log A+\varkappa -1.6449, \quad \SADD(T_A) \approx  \log A+\varkappa-1.
\]

Unfortunately, neither the limiting average overshoot $\varkappa$ nor the limiting average ``exponential'' overshoot $\zeta$ are computable exactly. Monte Carlo simulations with $10^6$ trials have been used to estimate the two as $\varkappa\approx1.255$ and $\zeta\approx0.426$ with the standard error less than $10^{-3}$.  Specifically, these estimates were obtained from the formulas
\[
\varkappa=\frac{\Eb_0 S_1^2}{2\Eb_0S_1}-\sum_{k=1}^\infty\frac{1}{k}\Eb_0[S_k^-], \quad
\zeta=\frac{1}{I}\exp\left\{-2\sum_{k=1}^\infty\frac{1}{k}\biggl[\Pb_0(S_k\le0)+\Pb_\infty(S_k>0)\biggr]\right\},
\]
where $S_k=\sum_{j=1}^k\log\Lambda_j$, $S_k^-=\min(0, S_k)$, and $I=\Eb_0[\log\Lambda_1]$ (see, e.g., Woodroofe~\cite{Woodroofe:Book82}). The first fraction in the first formula is computable analytically. The only issue is the infinite sum.
To evaluate this sum it was first truncated at $10^5$.
An extreme-value-theoretic argument shows that the weight of the dropped tail is of order of the machine precision.
This makes it safe to assume that the sum of the first $10^5$ terms is effectively equal to the original infinite sum.
The second source of errors is the expectations under the sum.
These expectations are not computable analytically, and therefore Monte Carlo simulations were used:
we generated $10^6$ trajectories of $S_1,S_2,\ldots,S_{100000}$, and performed averaging across the trajectories to find $\Eb_0[S_k^-]$ for each $k=1,2,\ldots,100000$. The sum in the second formula was evaluated in a similar manner.

Despite the fact that in the example considered the distributions $\Qb_{\mathrm{st}}(x)$ and $\widetilde{\Qb}(x)$ are obtainable exactly, neither the quasi-stationary distribution, required for the SRP procedure, nor the conditional average delay to detection $\Eb_\nu(T-\nu|T>\nu)$ for $\nu\ge0$ and the ARL to false alarm seem feasible to get analytically. To overcome this difficulty, these quantities were computed numerically, using the approach undertaken by Moustakides~et~al.~\cite{Moustakidesetal-SS09} with the number of breakpoints set at $5\times10^4$, high enough to ensure a relative error in the order of a fraction of a percent.

Specifically, let $j=\{\infty,0\}$.   For $A>0$, $r\ge0$, and $\nu \ge 0$, define $\phi_j(r)=\Eb_j T_A^r$, $\delta_\nu(r)=\Eb_\nu[(T_A^r-\nu)^+]$, $p_\nu(r)=\Pb_\infty(T_A^r>\nu)$, and  $F_j(x)=\Pb_j(\Lambda_1\le x)$. Using the Markov property of the statistic $R_n^r$, the following integral equations and recursions for operating characteristics are obtained by Moustakides~et~al.~\cite{Moustakidesetal-SS09}:
\[
\begin{aligned}
\phi_j(r) & =1+\int_0^A\phi_j(x)\left[\dfrac{\partial}{\partial x}F_j\left(\frac{x}{1+r}\right)\right]dx, \quad j=0,\infty,
\\
\delta_\nu(r) & =\int_0^A\delta_{\nu-1}(x)\left[\dfrac{\partial}{\partial x} F_\infty\left(\dfrac{x}{1+r}\right)\right]dx,\quad \nu\ge 1,
\\
p_\nu(r) &=\int_0^Ap_{\nu-1}(x)\,\frac{\partial}{\partial x}F_\infty\left(\frac{x}{1+r}\right)dx, \quad \nu\ge 1
\end{aligned}
\]
with the initial conditions $\delta_0(r)=\Eb_0 T_A^r=\phi_0(r)$ and $p_0(r) =1$. These integral equations yield the ARL to false alarm $\Eb_\infty T_A^r$ and the sequence of average detection delays $\Eb_\nu(T_A^r-\nu | T_A^r > \nu) = \delta_{\nu}(r)/p_\nu(r)$, $ \nu=0,1, \dots$ as functions of the starting point $r\ge0$. The distribution function $F_\infty(x)$ is defined in \eqref{Finfty} and $F_0(x)=[x/(1+x)]^2$, $ x \ge 0$.

In order to implement the SRP procedure $T_A^{\Qb_A}$ as well as to evaluate its performance, we need to compute the quasi-stationary distribution $\Qb_A(x)$, which  satisfies the integral equation~\eqref{InteqQ}. The ARL to false alarm $\Eb_\infty T_A^{\Qb_A}$ and the average detection delay $\Eb_0 T_A^{\Qb_A}$ of the SRP procedure are then computed as
\[
\Eb_j T_A^{\Qb_A}=\int_0^A \Eb_j [T_A^r]\,d\Qb_A(r)=\int_0^A \phi_j(r)\,d\Qb_A(r),\quad j=\infty,0.
\]
We recall that the SRP procedure is the equalizer, so that $\Eb_\nu(T_A^{\Qb_A}-\nu | T_A^{\Qb_A} > \nu)=  \Eb_0 T_A^{\Qb_A}$ for all $\nu\ge1$.

Finally, by Moustakides~et~al.~\cite{Moustakidesetal-SS09}, the lower bound $\Jc(T_A)$ given in Lemma~\ref{Lem2} is computed as $\Jc(T_A)= \psi(0)/\phi_\infty(0)$, where $\psi(r)$ is the solution of the integral equation
\[
\psi(r)=\phi_0(r)+\int_0^A\psi(r)\, \brcs{\frac{\partial}{\partial x}F_\infty\left(\frac{x}{1+r}\right)} \,dx.
\]

The above integral equations allow us to compute numerically operating characteristics of both SR--$r$ and SRP procedures as well as the mean of the quasi-stationary distribution $\mu_A$.

\begin{figure}[h]
    \centering
    \includegraphics[width=0.9\textwidth]{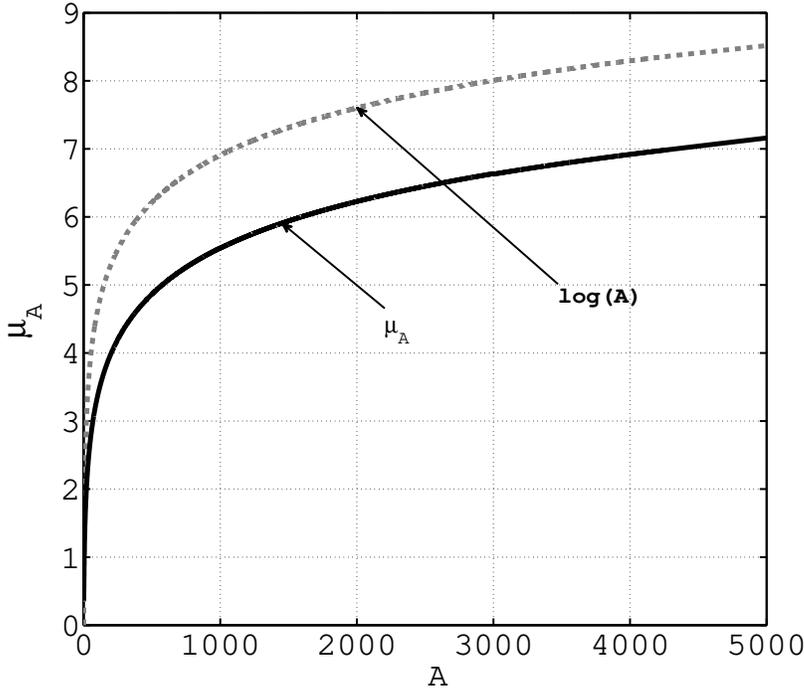}\vspace{-3mm}
    \caption{The mean $\mu_A$ of the quasi-stationary distribution $\Qb_A(x)$ as a function of the detection threshold $A$. The $\log(A)$ function is plotted to demonstrate that  $\mu_A=\log A +O(1)$.}
    \label{fig:mu_A_A}
\end{figure}

At this point the only unresolved question is that of how to choose $r$. To this end, several options are proposed by Moustakides~et~al.~\cite{Moustakidesetal-SS09}; one of the options is $r=\mu_A$. Recall that Theorem~\ref{Th4} requires (a) $r=o(A)$ as $A \to \infty$ and (b) $\SADD(T_A^r) = \ADD_\infty(T_A^r)$. If $r=\mu_A$, then condition (a) is satisfied, since by \eqref{meanQSupper} and \eqref{meanQS},  $\mu_A \le O(\log A)$ and $\mu_A=\log A +O(1)$. This is also illustrated in Figure~\ref{fig:mu_A_A}, which shows that the inequality $\mu_A\le O(\log A)$ and the equality $\mu_A=\log A +O(1)$ are indeed  satisfied.

The condition (b) is also satisfied even for small values of the ARL to false alarm, as can be seen from Figure~\ref{fig:OC_vs_k} which shows how $\Eb_\nu(T-\nu|T>\nu)$ evolves as $\nu$ runs from 0 to 10 for the SRP procedure and for the SR--$r$ procedure with $r=\mu_A$. The ARL to false alarm is about 100 for both procedures. Observe that the SR--$r$ procedure attains supremum at infinity, i.e., as $\nu\to\infty$. Also, the stationary regime kicks in as early as at $\nu=6$, and this is for $\Eb_\infty[T]\approx 100$. In addition, Figure~\ref{fig:OC_vs_k} illustrates Theorem~\ref{Th5} -- $\ADD_\infty(T)$ is indeed the smallest for the SR procedure, while the difference is small. We iterate that it is easily shown that the log-concavity conditions of Theorem~\ref{Th5} hold in the example considered, i.e., $\log [\Pb_\infty(\log\Lambda_1 \le x)]$ and $\log [\Pb_0(\log\Lambda_1 \le x)]$ are concave functions.

Table~\ref{tab:OC_vs_ARL} provides  values of the supremum average delay to detection $\SADD$ and the lower bound $\Jc(T_A)$ versus the ARL to false alarm $\Eb_\infty[T]$. Also presented in parentheses are the corresponding theoretical predictions made based on the asymptotic approximations~\eqref{SADDapprox} and~\eqref{ARLapprox}. Its is seen that the approximations for the ARL to false alarm are fairly accurate even for small values of the ARL such as 50, while the approximations for the $\SADD$ and the lower bound become accurate for the moderate false alarm rate ($\ARL=500$ and higher). 

\begin{figure}[htb]
    \centering
    \includegraphics[width=0.95\textwidth]{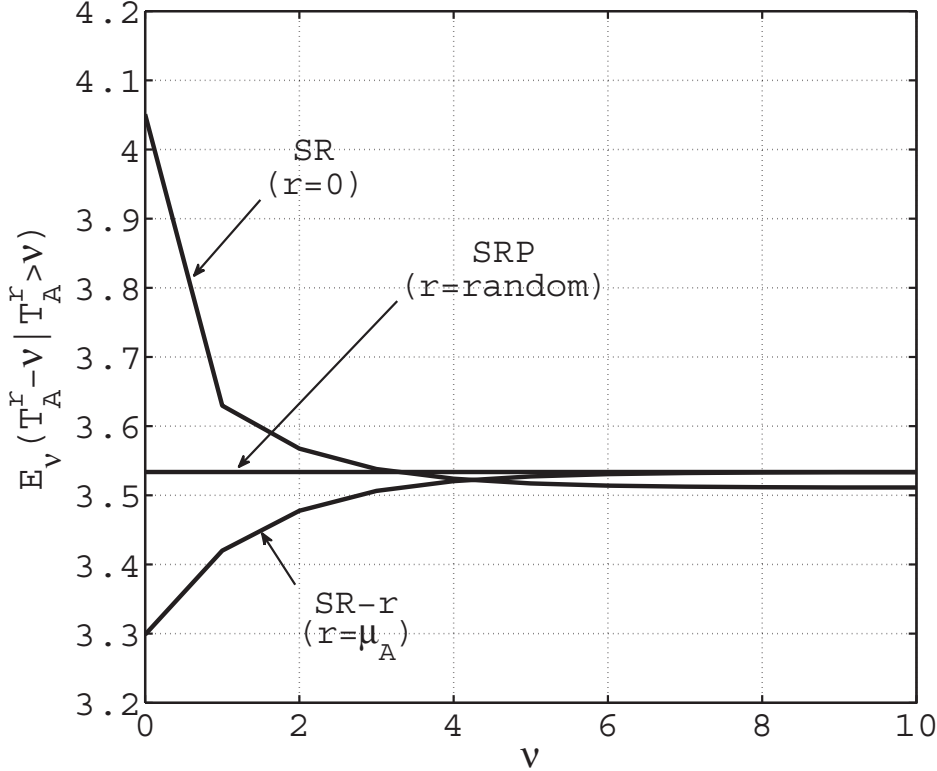}\vspace{-3mm}
    \caption{Results of numerical evaluation of the conditional average detection delay vs. changepoint $\nu$ of the SR, SRP and SR--$r$ {\rm(}$r=\mu_A${\rm)} procedures. The ARL to false alarm $\Eb_{\infty}[T]\approx 100$.}
    \label{fig:OC_vs_k}
\end{figure}

\begin{sidewaystable}[p]
    \centering
    \caption{Summary of the results of numerical evaluation of operating characteristics of the SR, SRP and SR--$r$ procedures. Numbers in parentheses are computed using the asymptotic approximations.}
    \begin{tabular}{|l||c||c|c|c|c|c|}
    \hline Test &$\gamma$ &50 &100 &500 &1000 &10000\\
        \hline
    \hline
        \multirow{3}[8]{0.9cm}{SR}
        &$A$ &21.0 &42.0 &212.0 &424.5 &4256.0\bigstrut\\\cline{2-7}
        & ARL to false alarm  &50.412 (49.342) &99.832 (98.684) &499.866 (498.12) &999.797 (997.415) &9999.675 (10000.0) \bigstrut\\\cline{2-7}
        &$\SADD$ &3.407 (3.312) &4.051 (4.005) &5.622 (5.615) &6.309 (6.308) &8.607 (8.611)\bigstrut\\\cline{2-7}
        \hline
    \hline
        \multirow{3}[8]{0.9cm}{SRP}
        &$A$ &21.5 &43.0 &213.5 &426.5 &4259.0 \bigstrut\\\cline{2-7}
        & ARL to false alarm  &49.635 (48.48) &99.664 (98.431) &499.424 (497.595) &999.87 (997.404) &9999.81 (10000.066)\bigstrut\\\cline{2-7}
        &$\SADD$ &2.942 (2.668) &3.534 (3.361) &5.021 (4.97) &5.692 (5.663) &7.965 (7.966) \bigstrut\\
        \hline
    \hline
        \multirow{4}[8]{0.9cm}{SR--$r$}
        &$A$ &21.5 &43.0 &213.5 &426.5 &4259.0 \bigstrut\\\cline{2-7}
        &$r=\mu_A$ &2.037 &2.603 &4.052 &4.711 &6.982 \bigstrut\\\cline{2-7}
        & ARL to false alarm &49.554 (48.48) &99.582 (98.431) &500.52 (497.595) &999.792 (997.404) &9999.735 (10000.066)\bigstrut\\\cline{2-7}
        &$\SADD$ &2.942 (2.668) &3.534 (3.361) &5.023 (4.97) &5.692 (5.663) &7.965 (7.966) \bigstrut\\\hline
    \hline
        &Lower Bound &2.939 (2.668) &3.523 (3.361) &5.017 (4.97) &5.688 (5.663) &7.965 (7.966)\\
        \hline
    \end{tabular}
    \label{tab:OC_vs_ARL}
\end{sidewaystable}

Another possible way of starting the SR$-r$ procedure is from the value of $r$ for which the average detection delay at the point $\nu=0$ is equal (at least approximately) to the $\ADD_\infty$ (i.e., in the steady-state mode), as has been proposed in Section~\ref{ss:ADD}. In the asymptotic setting this is equivalent to  finding a point $r=r^*$ for which $C_{\infty}$ is equal to $C_r$. Clearly, $r^*$ is a {\em fixed} number that does not depend on $A$ since $C_{\infty}$ does not depend on $r$ and $A$. Using \eqref{C0infty},
 we obtain the transcendental equation
\begin{align*}
\frac{1+r^*}{r^*}\log(1+r^*)&=\frac{\pi^2}{6},
\end{align*}
and the solution is $r^{*}\approx 1.98$.

\begin{figure}[htb]
    \centering
    \includegraphics[width=0.95\textwidth]{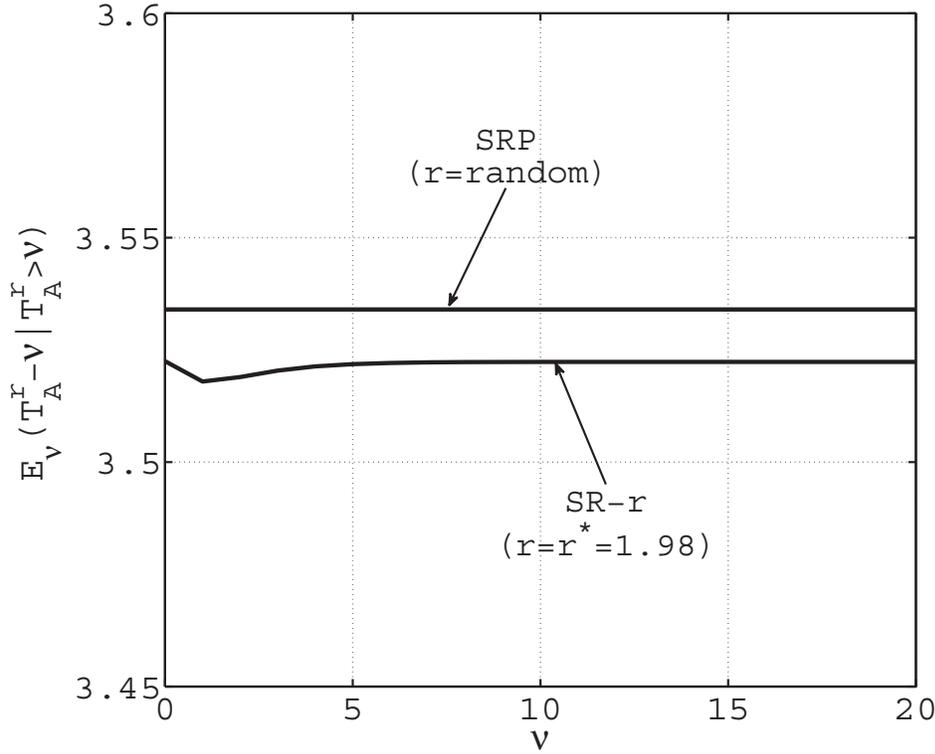}
    \caption{Conditional average detection delay vs. changepoint $\nu$ for the SRP procedure and for the SR--$r$ procedure with $r=r^*= 1.98$. The ARL to false alarm $\Eb_{\infty}[T]\approx 100$.}
    \label{fig:SRP_SR_r_ADDk_vs_k_ARL_100}
\end{figure}

Figure~\ref{fig:SRP_SR_r_ADDk_vs_k_ARL_100} shows the average delay to detection $\Eb_\nu[T-\nu|T>\nu]$ versus the changepoint $\nu$ for the SR$-r$ procedure with $r=r^*=1.98$ and for the SRP procedure. Observe that for the SR$-r$ procedure the average delay at $\nu=0$ is equal to that at infinity, as was planned. More importantly, the point $\nu=0$ is the worst (supremum) point (along with large $\nu$). Also, it can be seen that the SR$-r$ procedure is {\em uniformly} (i.e., for all $\nu\ge 0$) better than the SRP procedure, although in this example the difference is practically negligible. We also note that this initialization is better than starting off at the mean of the quasi-stationary distribution, while the difference in performance is very small -- the $\SADD$ is equal to $3.54$ for $r=\mu_A$ and $3.52$ for $r=r^*$. This allows us to conclude that the SR$-r$ is robust with respect to the initialization point in a certain range.

\section{Conclusions}

We considered three different versions  of the Shiryaev--Roberts procedure, with the difference being the starting point -- the conventional SR procedure, where $R_0=0$, Pollak's modification of this procedure, where $R_0$ is sampled from the quasi-stationary distribution of the SR statistic, and  a generalization where $R_0=r\ge0$ is a specially designed  deterministic number, proposed by Moustakides~et~al.~\cite{Moustakidesetal-SS09}. For each of the procedures we derived asymptotic formulas for operating characteristics when the threshold $A$ is high and showed that asymptotically when the ARL to false alarm is large the SR--$r$ and SRP procedures are both asymptotically third-order minimax.  We emphasize that third-order asymptotic optimality of the SR--$r$ procedure has been established under the conjecture that the worst changepoint is at infinity, which is justified numerically for several examples, including the one considered in Section~\ref{s:Example}. Unfortunately, we have not been able to prove this conjecture analytically. In addition, we performed a comparative efficiency analysis of the detection procedures to verify the accuracy of the asymptotic approximations and demonstrated the proximity of the latter to the real values for a specific example. The results of numerical analysis allow us to conclude that (in the minimax sense) performance of the SR--$r$ procedure that starts with the mean of the quasi-stationary distribution as well as with a point that equates the average detection  delay at zero and infinity  is almost indistinguishable from that of the SRP procedure.

 \section*{Acknowledgement}

We are grateful to a referee for useful suggestions that have improved the presentation in the article.

\Appendix

\section{Auxiliary Results and Proofs}

\begin{lemma}\label{LemSmax} 
Let $Y_1, Y_2, \dots$ be i.i.d.\ with $\Eb Y_1=0$ and $\Eb Y_1^2=\sigma^2<\infty$. Let $S_n = Y_1+\cdots+Y_n$. Then for all $\varepsilon >0$
\[
N \, \Pb\brc{\max\limits_{1\le n \le N}|S_n| > \varepsilon N}  \xra[N\to\infty]{} 0.
\]
\end{lemma}

\proof
Applying Doob's maximal submartingale inequality to the submartingale $S_n^2$, we obtain
\[
\begin{aligned}
\Pb \brc{\max_{n \le N} |S_n| \ge \varepsilon N} & \le  \frac{1}{(\varepsilon N)^2} \Eb\brcs{S_N^2 \Ind{\max\limits_{n \le N} S_{n} \ge  \varepsilon N}}
\\
&=\frac{1}{\varepsilon^2 N} \Eb\brcs{\brc{\frac{S_N^2}{N}} \Ind{\max\limits_{n \le N} S_{n} \ge  \varepsilon N}} .
\end{aligned}
\]
First, it follows that
\[
\Pb\brc{\max_{n\le N} |S_n| \ge \varepsilon N} \le \frac{\sigma^2}{\varepsilon^2 N}\xrightarrow[N\to\infty]{} 0 .
\]
Now, we show that
\[
 \Eb\brcs{\brc{\frac{S_N^2}{N}} \Ind{\max\limits_{n \le N} S_{n} \ge  \varepsilon N}}
\xrightarrow[N\to\infty]{} 0,
\]
which implies that $\Pb\big(\max\limits_{n\le N}|S_n| > \varepsilon N\big) = o(1/N)$ as $N\to\infty$, i.e., the desired result.

By the second moment condition, $\Eb(S_N^2/N)=\sigma^2<\infty$.  Hence, by the Central Limit Theorem,
\[
\frac{S^2_N}{N \sigma^2}\xrightarrow[n\to\infty]{\rm law}\xi^2,
\]
where $\xi$ is a standard normal random variable.

Finally, for any $L<\infty$ we have
\begin{align*}
\begin{split}
\Eb\brc{\frac{S_N^2}{N}\Ind{\max\limits_{n\le N} S_n > \varepsilon N}}
&=
\Eb\brcs{\brc{L\wedge\frac{S_N^2}{N}}\Ind{\max\limits_{n\le N} S_n > \varepsilon N}}+{}\\
&\qquad\qquad\qquad{}+\Eb \brcs{\brc{\frac{S_N^2}{N}-L\wedge\frac{S_N^2}{N}}\Ind{\max\limits_{n\le N} S_n > \varepsilon N}}
\\
&\le L\Pb\brc{\max\limits_{n\le N} S_n > \varepsilon N }+ \Eb \brc{\frac{S_N^2}{N}-L\wedge\frac{S_N^2}{N}}
\\
&\le \frac{L\sigma^2}{\varepsilon^2 N}+\sigma^2-\Eb \brc{L\wedge\frac{S_N^2}{N}}
\\
&\xrightarrow[N\to\infty]{}\sigma^2-\Eb \brc{L\wedge\xi^2\sigma^2} \xrightarrow[L\to\infty]{}\sigma^2(1-1)=0.
\end{split}
\end{align*}
\endproof

\renewcommand{\proof}{{\em Proof of Theorem~\ref{Th1}.}}

\begin{proof} 
It follows from Pollak~\cite{Pollak:AS87} that for the SR procedure
\[
\Eb_\infty T_A=(A/\zeta)(1+o(1)) , \quad A\to\infty.
\]

Since $R_n^r= r e^{S_n}+ R_n \ge R_n$, we have
\begin{equation} \label{ARLupper}
\Eb_\infty T_A^r \le \Eb_\infty T_A = (A/ \zeta)(1+o(1)) \quad \text{for any} ~ r \ge 0.
\end{equation}

For some positive $m$, define
\[
M=\inf\{n\colon r e^{S_n} \ge m\}.
\]

Observe that
\[
\Eb_\infty T_A^r  = \Eb_\infty(T_A^r; T_A^r <M) + \Eb_\infty(T_A^r; T_A^r \ge M)
\]
and
\[
R_{T_A^r} =R^r_{T_A^r} -r\exp\{S_{T_A^r}\} \ge A-m \quad \text{on $\{T_A^r <M\}$}.
\]
Hence, $T_{A-m} \le T_A^r$ on $\{T_A^r <M\}$, which implies that
\[
\Eb_\infty (T_A^r; T_A^r <M) \ge \Eb_\infty(T_{A-m}; T_A^r <M).
\]
Therefore, we have the following chain of equalities and inequalities:
\begin{align*}
\Eb_\infty T_A^r & = \Eb_\infty(T_A^r; T_A^r <M) + \Eb_\infty(T_A^r; T_A^r \ge M)
\\
& \ge  \Eb_\infty(T_{A-m}; T_A^r <M) + \Eb_\infty(T_A^r; T_A^r \ge M)
\\
& = \Eb_\infty T_{A-m} + \Eb_\infty(T_A^r-T_{A-m}; T_A^r \ge M)
\\
& \ge \Eb_\infty T_{A-m} + \Eb_\infty(T_A^r-T_{A-m}; T_{A-m} > T_A^r \ge M)
\\
& =  \Eb_\infty T_{A-m} - \Eb_\infty(T_{A-m}- T_A^r; T_{A-m} > T_A^r \ge M)
\\
& =  \Eb_\infty T_{A-m} - \Eb_\infty(T_{A-m}- T_A^r | T_{A-m} > T_A^r \ge M) \Pb_\infty (T_{A-m} > T_A^r \ge M)
\\
& \ge \Eb_\infty T_{A-m} - \Eb_\infty T_{A-m} \Pb_\infty(M<\infty) ,
\end{align*}
where the last inequality stems from $\{T_{A-m} > T_A^r \ge M\} \subseteq \{M < \infty\}$ (so that $ \Pb_\infty\{T_{A-m} > T_A^r\ge M\} \le \Pb_\infty\{M <\infty\}$) and from
\[
\begin{aligned}
\Eb_\infty(T_{A-m}- T_A^r | T_{A-m} > T_A^r \ge M)  & = \Eb_\infty\{\Eb_\infty \brcs{ R_{T_{A-m}} - T_A^r | T_A^r} | T_{A-m} > T_A^r \ge M\}
\\
& \le \Eb_\infty R_{T_A-m} = \Eb_\infty T_{A-m}.
\end{aligned}
\]

Note that $e^{S_n}$ is a nonnegative $\Pb_\infty$-martingale with mean 1, so that
\begin{equation}\label{Martineq}
\Pb_\infty (M<\infty) = \Pb_\infty \brc{\inf\{n\colon e^{S_n} \ge m/r\}< \infty} < r/m,
\end{equation}
and we obtain
\begin{equation}\label{ARLlower}
\begin{aligned}
\Eb_\infty T_A^r&\ge\Eb_\infty T_{A-m}(1-r/m)=\frac{A-m}{\zeta}(1-r/m)(1+o(1))
\\
&=(A/\zeta)(1-m/A)(1-r/m) (1+o(1)) .
\end{aligned}
\end{equation}
Let $r_A^* \to \infty$ and $m=m_A\to \infty$ so that $r_A^*/m_A \to 0$ and $m_A/A \to 0$ (which can always be arranged). Then, uniformly in $ 0 \le r \le r_A^*$,
\[
\Eb_\infty T_A^r \ge (A/\zeta) (1+o(1)), \quad A \to \infty,
\]
which along with the reverse inequality~\eqref{ARLupper} proves asymptotic equality~\eqref{ARLrA} whenever $r_A^*=o(A)$ as $A\to\infty$, and if $r^*$ does not depend on $A$ the result obviously holds.

Similar to \eqref{Martineq},
\[
\Pb_\infty (M<\infty | R_0^{\Qb_A}=x)<x/m_A.
\]
Thus, for the SRP procedure, by conditioning on $R_0^{\Qb_A}$, we obtain
\[
\Pb_\infty (M <\infty) = \int_0^A \Pb_\infty (M<\infty | R_0^{\Qb_A}=x) d\Qb_A(x) \le \frac{1}{m_A} \int_0^A x d\Qb_A(x) = \frac{\mu_A}{m_A},
\]
where $\mu_A = \Eb R_0^{\Qb_A}= \int_0^A x d \Qb_A(x)$ is  the mean of the quasi-stationary
distribution. By \eqref{meanQSupper},  $\mu_A/m_A  \le O(m_A^{-1}\log A)$ and to obtain~\eqref{ARLSRP} it suffices to take $m_A=A^{1/2}$ (say).
\end{proof}

\renewcommand{\proof}{{\em Proof of Lemma~\ref{Lem1}.}}

\begin{proof}
For any $\nu \ge 0$, the SR--$r$ statistic can be written as
\begin{align*}
R_{\nu+n}^r & = (1+ R_\nu^r) \prod_{i=\nu+1}^{\nu+n} \Lambda_i + \brc{\prod_{i=\nu+1}^{\nu+n} \Lambda_i}  \sum_{j=\nu+1}^{\nu+n-1} \brc{\prod_{k=\nu+1}^{j} \Lambda_k^{-1}}
\\
& = \brc{ \prod_{i=\nu+1}^{\nu+n} \Lambda_i} \brc{1+ R_\nu^r + \sum_{j=\nu+1}^{\nu+n-1} e^{-S_j^{\nu+1}}}.
\end{align*}

Thus, we have
\begin{align*}
\log R_{\nu+n}^r & = \sum_{i=\nu+1}^{\nu+n} \log \Lambda_i +\log \brc{1+ R_\nu^r + \sum_{j=\nu+1}^{\nu+n-1}
e^{-S_j^{\nu+1}}}
\\
& = \sum_{i=\nu+1}^{\nu+n} Z_i+\log \brc{1+ R_\nu^r + V_{\nu,n}} ,
\end{align*}
where $V_{\nu,n} = \sum_{j=\nu+1}^{\nu+n-1} e^{-S_j^{\nu+1}}$.

On $\{T_A^r >\nu\}$ the stopping time $T_A^r$ can be written as
\begin{equation}\label{STSRr}
T_A^r = \inf \set{n\ge1\colon \sum_{i=\nu+1}^{\nu+n} Z_i+ \log \brc{1+\frac{V_{\nu,n}}{1+R_\nu^r} }\ge \log \brc{\frac{A}{1+R_\nu^r}} }.
\end{equation}
Note that on $\{R_\nu^r <A/N_A\}$,
\[
\log \brc{\frac{A}{1+R_\nu^r}} > \log \brc{\frac{A N_A}{N_A+A}} =\log N_A +o(1) \to \infty \quad \text{as $A \to \infty$} .
\]

Therefore, nonlinear renewal theory can be applied to the sequence
\[
\sum_{i=\nu+1}^{\nu+n} Z_i +\log \brc{1+\frac{V_{\nu,n}}{1+R_\nu^r}}, \quad n \ge 1.
\]
Note also that
\[
0 < \log \brc{1+\frac{V_{\nu,n}}{1+R_\nu^r}} < \log\brc{1+ V_{\nu,n}},
\]
so that the sequence
\[
\log \brc{1+\frac{V_{\nu,n}}{1+R_\nu^r}} , \quad n \ge 1
\]
satisfies the conditions of the nonlinear renewal theorem of Woodroofe~\cite[Theorem~4.5]{Woodroofe:Book82} uniformly in $R_\nu^r$.  Indeed, the sequence $\{V_{\nu, n}\}_{n\ge 1}$ is slowly changing and converges  $\Pb_\nu-$a.s. (as $n \to \infty$) to the finite random variable $V_{\nu, \infty}= \sum_{j=\nu+1}^{\infty} e^{-S_j^{\nu+1}}$.  The nonlinear renewal theorem yields the following asymptotic approximation:
\begin{equation*}
\begin{aligned}
\Eb_\nu (T_A^r - \nu | T_A^r >\nu, R_\nu^r)  & = \frac{1}{I}  \Bigg\{\log A + \varkappa - \log(1+R_\nu^r)-{}
\\
 &{}\qquad\qquad -\Eb_\nu\brcs{\log\brc{1+ \frac{V_{\nu, \infty}}{1+R_\nu^r}} \Bigg | T_A^r >\nu, R_\nu^r} \Bigg\} + o(1) ,
\end{aligned}
\end{equation*}
where $o(1) \to 0$ as $A\to \infty$ uniformly on $\{N_A \le \nu < \infty, R_\nu^{r}< A/N_A, 0 \le r < \infty\}$.

Note that all the necessary conditions of this theorem hold trivially.
The only condition that requires checking is the following: For some $\varepsilon >0$,
\begin{equation} \label{2nd1}
(\log A) \, \Pb_\nu \brc{T_A^r -\nu \le \varepsilon I^{-1} \log A |T_A^r > \nu, R_\nu^r } \xra[A\to\infty]{} 0  \quad \text{on $\{R_\nu^r < A/N_A\}$ } .
\end{equation}

Let $L=L_{A,\varepsilon} = (1-\varepsilon)I^{-1} \log A$, $p_\nu(A,\varepsilon) = \Pb_\nu(T_A^r \le \nu + L | T_A^r > \nu, R_\nu^r)$, and $\Bc_{A,\varepsilon}= \set{T_A^r \le \nu +L_{A,\varepsilon}}$. Changing the measure $\Pb_\infty \mapsto  \Pb_\nu$, we obtain that for any $C >0$ and $\varepsilon \in (0,1)$
\[
\begin{aligned}
\Pb_\infty(T_A^r \le \nu + L | T_A^r > \nu, R_\nu^r) &  =
\Eb_\nu\brcs{\exp\set{-S_{T_A^r}^{\nu+1}} \Ind{\Bc_{A,\varepsilon} } | T_A^r > \nu, R_\nu^r}
\\
& \ge \Eb_\nu\Big[\exp\set{-S_{T_A^r}^{\nu+1}} \Ind{\Bc_{A,\varepsilon},S_{T_A^r}^{\nu+1} \le C}  | T_A^r > \nu, R_\nu^r\Big]
\\
& \ge e^{-C} \Big [\Pb_\nu\brc{T_A^r \le \nu + L | T_A^r > \nu, R_\nu^r} -
\\
& \quad \Pb_\nu \Big(\max_{1\le n \le L} S_{\nu+n}^{\nu+1} > C | T_A^r > \nu, R_\nu^r \Big)\Big ].
\end{aligned}
\]
Setting $C=(1+\varepsilon) I L$ and noting that
\[
\begin{aligned}
\Pb_\nu \brc{\max_{1\le n \le L} S_{\nu+n}^{\nu+1} > C | T_A^r > \nu, R_\nu^r } & = \Pb_0 \brc{\max_{1\le n \le  L} S_{n} > C},
\\
\Pb_\nu\brc{T_A^r \le \nu + L | T_A^r > \nu, R_\nu^r} &= \Pb_\infty \brc{\nu < T_A^r \le \nu + L | T_A^r > \nu, R_\nu^r}
\\
& =  \Pb_\infty \brc{\max_{1\le n \le L} R_{\nu+ n}^r \ge A | T_A^r > \nu, R_\nu^r }
\\
&=  \Pb_\infty \brc{\max_{1\le n \le L} R_{\nu+ n}^r \ge A | R_\nu^r },
\end{aligned}
\]
we obtain
\[
p_\nu(A,\varepsilon) \le \alpha_\nu(A,\varepsilon)  + \beta(A,\varepsilon) ,
\]
where
\[
\alpha_\nu(A,\varepsilon) = e^{(1+\varepsilon) I L}\Pb_\infty \brc{\max_{1\le n \le L} R_{\nu+ n}^r \ge A | R_\nu^r }, ~~
\beta(A,\varepsilon) = \Pb_0 \brc{\max_{1\le n \le  L} S_{n} > (1+\varepsilon) I L}.
\]

Note that $\{R_{\nu+ n}^r\}_{n\ge 1}$ is a non-negative $\Pb_\infty$-submartingale with mean $\Eb_\infty(R_{\nu+n}^r | R_\nu^r)=n+R_\nu^r$. By Doob's submartingale inequality,
\[
\begin{aligned}
\alpha_\nu(A,\varepsilon) & \le \brc{e^{(1+\varepsilon) I L}} \frac{L + R_\nu^r}{A}
= \brc{e^{(1-\varepsilon^2) \log A}} \frac{(1-\varepsilon) I^{-1} \log A + R_\nu^r}{A}
\\
& =   \frac{(1-\varepsilon) I^{-1} \log A + R_\nu^r}{A^{\varepsilon^2}} = o(1/\log A) \quad \text{on} \, \, R_{\nu}^r < A/N_A.
\end{aligned}
\]

It remains to show that $\beta(A,\varepsilon)=o(1/\log A)$ as $A\to\infty$ for some $0<\varepsilon<1$.  Note that
\[
\begin{aligned}
\beta(A,\varepsilon) & =\Pb_0 \brc{\max_{1\le n \le L} (S_n-I L) > \varepsilon I L} \le \Pb_0 \brc{\max_{1\le n \le L} (S_n-I n) > \varepsilon I L}
\end{aligned}
\]
By Lemma~\ref{LemSmax},
\[
L \, \Pb_0 \brc{\max_{1\le n \le L} (S_n-I n) > \varepsilon I L} \xra[L\to\infty]{} 0,
\]
so that $\beta(A,\varepsilon) = o(1/L)=o(1/\log A)$.
Thus, condition~\eqref{2nd1} holds, the asymptotic approximation~\eqref{AAAnu} follows and the proof is complete.
\end{proof}

\renewcommand{\proof}{{\em Proof of Theorem~\ref{Th2}.}}

\begin{proof}
Obviously, $\ADD_\infty(T_A^r) =\Eb_0 T_A^{\Qb_A}$ for any fixed $r \ge 0$ and any $A>0$, so that it suffices to prove asymptotic expansion \eqref{ADDinfty} only for $ \ADD_\infty(T_A^r)$.

Write $L_A=A/N_A$. Note that $L_A^{-1}=o(1/\log A)$ as $A\to\infty$. Obviously,
\begin{align} \label{Enu}
\begin{split}
\Eb_\nu(T_A^r -\nu |T_A^r >\nu)
&=
\Eb_\nu(T_A^r -\nu |T_A^r >\nu,  R_\nu^r< L_A ) \Pb_\infty (R_\nu^r <L_A|T_A^r >\nu)+{}
\\
&{}+\Eb_\nu(T_A^r -\nu |T_A^r >\nu,  R_\nu^r \ge L_A) \Pb_\infty (R_\nu^r \ge L_A |T_A^r >\nu)
\\
&=\Eb_\nu(T_A^r -\nu |T_A^r >\nu,  R_\nu^r< L_A )+{}\\
&{}+\Pb_\infty (R_\nu^r \ge L_A |T_A^r >\nu)\Eb_\nu(T_A^r -\nu |T_A^r >\nu,  R_\nu^r \ge L_A)-{}
\\
&-\Pb_\infty (R_\nu^r \ge L_A |T_A^r >\nu) \Eb_\nu(T_A^r -\nu |T_A^r >\nu,  R_\nu^r < L_A).
\end{split}
\end{align}
Note that for large enough $A$
\begin{equation}\label{Enuupper}
\Eb_\nu (T_A^r -\nu |T_A^r >\nu, R_\nu^r) = \Eb_0 \brc{T_A^{R_\nu^r}} \le \frac{2 \log A}{I}
\end{equation}
and there exists $\nu_A^*$ such that for all $\nu>\nu_A^*$, by Kesten~\cite[Theorem~5]{Kesten:AM73},
\begin{equation} \label{Pnuineq}
\Pb_\infty(R_\nu^r \ge L_A | T_A^r >\nu) \le 2(1- \Qb_{st}(L_A)) = \frac{2}{L_A} (1+o(1)) \to 0 \quad \text{as $A \to \infty$}
\end{equation}
(when $r\ge A$ first condition on $R_1^r$, conditional on $T_A^r>\nu$), so that the second term in the last equality in \eqref{Enu} is $o(1)$. By Lemma~\ref{Lem1}, the first term in the last equality in \eqref{Enu} is equal to
\[
\begin{aligned}
\Eb_\nu(T_A^r -\nu |T_A^r >\nu,  R_\nu^r< L_A ) & = \frac{1}{I} \Bigg\{ \log A + \varkappa -
\Eb_\infty \brcs{\log(1+R_\nu^r) | T_A^r > \nu, R_\nu^r < L_A}
\\
& \quad -  \Eb_\nu\brcs{\log\brc{1+ \frac{V_{\nu,\infty}}{1+R_\nu^r}} \Bigg | T_A^r >\nu, R_\nu^r < L_A} \Bigg \} +o(1) .
\end{aligned}
\]

Now,
\[
\begin{aligned}
& \Eb_\infty \brcs{\log(1+R_\nu^r) | T_A^r > \nu, R_\nu^r < L_A}
\\
& = \frac{\Eb_\infty \brcs{\log(1+R_\nu^r) | T_A^r > \nu} - \Eb_\infty \brcs{\log(1+R_\nu^r) | T_A^r > \nu, R_\nu^r \ge L_A} \Pb_\infty (R_\nu^r \ge L_A| T_A^r > \nu)}{1-\Pb_\infty (R_\nu^r \ge L_A| T_A^r > \nu)}
\end{aligned}
\]
where by \eqref{Pnuineq}   $\Pb_\infty(R_\nu^r \ge L_A | T_A^r >\nu)=o(1/\log A)$. Also,  since $R_\nu^r < A$ on $\{T_A^r > \nu\}$,
\[
0 \le  \Eb_\infty \brcs{\log(1+R_\nu^r) | T_A^r > \nu, R_\nu^r \ge L_A}  \le \log (1+A),
\]
so that
\begin{equation}\label{ElogR}
\begin{aligned}
 \Eb_\infty \brcs{\log(1+R_\nu^r) | T_A^r > \nu, R_\nu^r < L_A} &=  \Eb_\infty \brcs{\log(1+R_\nu^r) | T_A^r > \nu}
 \\
 & \times \brcs{1+ \frac{\Pb_\infty(R_\nu^r \ge L_A|T_A^r >\nu)}{1-\Pb_\infty(R_\nu^r \ge L_A|T_A^r >\nu)}} +o(1)
 \\
& = \Eb_\infty [\log(1+R_\nu^r) | T_A^r >\nu] +o(1)
\end{aligned}
\end{equation}
(since $\Eb_\infty [\log(1+R_\nu^r) | T_A^r >\nu]<\log (1+A)$). By Remark~\ref{Rem3.1},  $V_{\nu, \infty}$ is independent of $R_\nu^r$ and distributed as  $V_\infty$ under $\Pb_0$. Since $\ADD_\infty(T_A^r)$ exists, it follows that
\[
\begin{aligned}
\ADD_\infty(T_A^r) & = \lim_{\nu\to\infty} \Eb_\nu(T_A^r-\nu|T_A^r >\nu)
\\
& = \frac{1}{I} \set{ \log A + \varkappa - \Eb\brcs{\log(1+R_\infty) } - \Eb\brcs{\log \brc{1+\frac{V_\infty}{1+R_\infty}}}} +o(1)
\\
& = \frac{1}{I} \set{ \log A + \varkappa - \Eb\brcs{\log \brc{1+R_\infty +V_\infty}} }+o(1) \quad \text{as $A \to \infty$}
\end{aligned}
\]
and the proof of~\eqref{ADDinfty} is complete.

It remains to prove the validity of asymptotic approximation \eqref{ADDzeroSRr}. Putting $\nu=0$ in \eqref{STSRr}, we obtain that  the stopping time of the SR$-r$ procedure can be written as
\[
T_A^r = \inf \set{n\ge1\colon S_n + \log \brc{1+ r +V_{0,n}} \ge \log A },
\]
where $V_{0,n} = \sum_{j=1}^{n-1} e^{-S_j}$ and $\{S_n\}_{n\ge 0}$ is the random walk with the drift $I$ (under $\Pb_0$).  The sequence
$\{\log(1+r+ V_{0,n})\}_{n\ge 1}$ is slowly changing and converges $\Pb_0-$a.s. to the random variable $\log(1+r+ V_{\infty})$ whose expectation is equal to $C_r$.
The crucial condition
\[
(\log A) \, \Pb_0 \brc{T_A^r \le \varepsilon I^{-1} \log A} \xra[A\to\infty]{} 0 \quad \text{for some}~\varepsilon >0
\]
holds in an analogous way that yields \eqref{2nd1}. Therefore, nonlinear renewal theory can be applied to yield
\begin{equation*}
\Eb_0[T_A^r]   = \frac{1}{I} (\log A + \varkappa - C_r ) + o(1)  \quad \text{as}~ A\to\infty,
\end{equation*}
and the proof is complete.
\end{proof}

\renewcommand{\proof}{{\em Proof of Theorem~\ref{Th3}.}}

\begin{proof}
Recall that $N_A=o(A/\log A)$. We have
\[
\begin{aligned}
\Jc(T_A)
&=
\frac{\sum_{\nu=0}^\infty \Eb_\nu(T_A-\nu|T_A>\nu) \Pb_\infty(T_A>\nu)}{\Eb_\infty T_A}
\\
&=\frac{\sum_{\nu=0}^{N_A} \Eb_\nu(T_A-\nu|T_A>\nu) \Pb_\infty(T_A>\nu)}{\Eb_\infty T_A} +{}\\
&{}\qquad\qquad+\frac{\sum_{\nu=N_A+1}^\infty \Eb_\nu(T_A-\nu|T_A>\nu) \Pb_\infty(T_A>\nu)}{\Eb_\infty T_A}.
\end{aligned}
\]
Since $\Eb_\infty T_A \ge A$ and since for a sufficiently large $A$,
\[
\sup_{\nu \ge 0} \Eb_\nu (T_A-\nu |T_A >\nu) =\Eb_0 T_A \le \frac{2\log A}{I},
\]
for the first term, denoted as $\Jc_1$, we have
\[
\Jc_1 \le \frac{N_A \Eb_0 T_A}{\Eb_\infty T_A} \le \frac{2 N_A \log A}{I A} \xra[A\to\infty]{} 0,
\]
so that $\Jc_1=o(1)$ as $A \to \infty$.

Write
\[
D_A=\frac{1}{I} \brc{\log A + \varkappa - C_\infty}.
\]
By Lemma~\ref{Lem1}, uniformly in $N_A \le \nu <\infty$
\[
\Eb_\nu(T_A-\nu|T_A >\nu) = D_A + o(1) \quad \text{as $A \to \infty$}.
\]
Therefore, for the second term (denote it as $\Jc_2$) we have
\[
\begin{aligned}
\Jc_2 & =  \frac{\sum_{\nu=1}^{N_A} D_A \Pb_\infty(T_A>\nu) + \sum_{\nu=N_A+1}^\infty [D_A+o(1)]  \Pb_\infty(T_A>\nu) }
{\Eb_\infty T_A}-{}\\
&{}\qquad\qquad\qquad\qquad\qquad\qquad-\frac{\sum_{\nu=1}^{N_A} D_A \Pb_\infty(T_A>\nu) }{\Eb_\infty T_A}
\\
& = D_A +o(1)-\frac{D_A \sum_{\nu=1}^{N_A} \Pb_\infty(T_A>\nu)}{\Eb_\infty T_A}  = D_A +o(1),
\end{aligned}
\]
where the last equality follows immediately from the fact that
\[
\frac{D_A \sum_{\nu=1}^{N_A} \Pb_\infty(T_A>\nu)}{\Eb_\infty T_A} \le \frac{(2\log A) N_A}{I A} \xra[A\to\infty]{}  0 ,
\]
where the inequality holds for a sufficiently large $A$. The required result follows.
\end{proof}

\renewcommand{\proof}{{\em Proof of Lemma~\ref{Lem3}.}}

\begin{proof}
We must show that $\Pb_\infty(R_{n+1}^{(\eta)} \le x| R_n^{(\eta)}=t, R_{n+1}^{(\eta)} <A)$ is a non-increasing function of  $t$.

For $x < A$ and $\theta >0$,
\begin{equation}\label{PR}
\begin{aligned}
\Pb \brc{ R_{n+1}^{(\eta)} \le x | R_n ^{(\eta)} = t, R_{n+1}^{(\eta)} < A} & = \frac{\Pb_\infty(\Lambda_1 \le x/(\eta+t))}{\Pb_\infty(\Lambda_1 \le A/(\eta+t))}
\\ & = \frac{F(\log x -\log(\eta+t))}{F(\log A -\log(\eta+t))}
\\
& = \frac{F(y-s)}{F(a-s)},
\end{aligned}
\end{equation}
where we used the notation $y=\log x $, $s = \log (\eta+t)$ and $a=\log A $.

If $\log F(x)$  is concave, then $(\log F(x))^{\prime\prime} \le 0$,  so
$(\log F(x))^\prime =f(x)/F(x)$  is a non-increasing function of $x$. Therefore, we have
\begin{align*}
\frac{\partial}{\partial s} \Pb \brc{ R_{n+1}^{(\eta)} \le x | R_n ^{(\eta)} = t(s), R_{n+1}^{(\eta)} < A} & = \frac{\partial}{\partial s} \brc{\frac{F(y-s)}{F(a-s)}}
\\
& = \frac{-f(y-s) F(a-s) + F(y-s) f(a-s)}{F^2(a-s)}
\\
& = \frac{F(y-s)}{F(a-s)} \brcs{\frac{f(a-s)}{F(a-s)} - \frac{f(y-s)}{F(y-s)} } \le 0 ,
\end{align*}
which implies that $\Pb(R_{n+1}^{(\eta)} \le x| R_n^{(\eta)}=t, R_{n+1}^{(\eta)} <A)$ is a non-increasing function of  $t$.
\end{proof}

\renewcommand{\proof}{{\em Proof of Theorem~\ref{Th5}.}}

\begin{proof}
Let $F$ denote the cdf of $\log \Lambda_1$, and let $\Pb$ denote probability when $F$ is the cdf of $\log \Lambda_1$, when observations are iid.  The following applies both for $F$ defined by $X_i\sim f$  as well as for $F$ defined by $X_i\sim g$.  Under the assumption of the theorem, $F$ is log-concave.

If $r_0<r_1$, then $R_n^{r_0}<R_n^{r_1}$ for all $n\ge0$, so that  $\{R_n^{r_1}\}$ crosses $A$ no later (and often earlier) than $\{R_n^{r_0}\}$. Hence,  $\Eb_\infty T_A^r$ is a decreasing function of $r$ and $A_\gamma^r$ is an increasing function of $r$. (Note that this is true in general, with no assumption of log-concavity.)

Now, note that $\ADD_\infty(T_{A}^{r}) = \ADD_\infty(T_{A}^{\Qb_{A}}) = \Eb_0 T_{A}^{\Qb_{A}}$ and, therefore, it suffices to show that $\Eb_0 T_{A}^{\Qb_{A}}$ and $\Eb_\infty T_{A}^{\Qb_{A}}$ are increasing functions of $A$.

Let $(M_n^{(1)})_{n\ge 0}$, $(M_n^{(2)})_{n\ge 0}$, and $(M_n^{(3)})_{n\ge 0}$ be Markov processes governed respectively by
\[
\begin{aligned}
\Pb\brc{M_{n+1}^{(1)} \le x | M_n^{(1)} =t} & = \Pb \brc{R_{n+1}^{(1)} \le x | R_n^{(1)} = t, R_{n+1}^{(1)} < A}
\\
\Pb\brc{M_{n+1}^{(2)} \le x | M_n^{(2)} =t} & = \Pb \brc{R_{n+1}^{(1)} \le x | R_n^{(1)} = t, R_{n+1}^{(1)} < \eta A}
\\
\Pb\brc{M_{n+1}^{(3)} \le x | M_n^{(3)} =t} & = \Pb \brc{R_{n+1}^{(\eta)} \le x | R_n^{(\eta)} = t, R_{n+1}^{(\eta)} < \eta A} ,
\end{aligned}
\]
where $\eta \ge 1$ and $R_n^{(\eta)}$ satisfies \eqref{Reta} with zero initial condition $R_0^{(\eta)}=0$.

Note that $M_{n}^{(3)}= \eta M_{n}^{(1)}$, and therefore, $M_{n}^{(1)}$ exits $[0,A)$ at the same time that $M_{n}^{(3)}$ exits $[0,\eta A)$ regardless of the distribution of $\Lambda_n$. It follows that $\Qb_A^{(1)}(x)=\Qb_{\eta A}^{(3)}(\eta x)$, where
$\Qb_A^{(1)}$ and $\Qb_{\eta A}^{(3)}$ are the $\Pb_\infty$-stationary distributions of  $M_{n}^{(1)}$ and  $M_{n}^{(3)}$, which are also the corresponding $\Pb_\infty$-quasi-stationary distributions of the Markov processes $R_n^{(1)}$ and $R_n^{(\eta)}$.

Let $0\le t<\eta A$. In a manner similar to~\eqref{PR},
\[
\begin{aligned}
\Pb \brc{ M_{n+1}^{(3)} \le x | M_n ^{(3)} = t, M_{n+1}^{(3)} < \eta A} & =  \frac{F(\log x -\log(\eta+t))}{F(\log A + \log \eta -\log(\eta+t))}
\\
& = \frac{F(y-s_\eta)}{F(a_\eta-s_\eta)}
\\
\Pb \brc{ M_{n+1}^{(2)} \le x | M_n ^{(2)} = t, M_{n+1}^{(2)} < \eta A} & =  \frac{F(\log x -\log(1+t))}{F(\log A + \log \eta -\log(1+t))}
\\
&= \frac{F(y-s_1)}{F(a_\eta-s_1)},
\end{aligned}
\]
where $y=\log x $, $s_\eta = \log (\eta+t)$ and $a_\eta=\log A + \log \eta$. Writing
\[
\Xi(s)= \frac{F(y-s)}{F(a_\eta-s)},
\]
we obtain that
\[
\begin{aligned}
\Pb \brc{ M_{n+1}^{(3)} \le x | M_n ^{(3)} = t, M_{n+1}^{(3)} < \eta A} & = \Xi(s_\eta) ,
\\
\Pb \brc{ M_{n+1}^{(2)} \le x | M_n ^{(2)} = t, M_{n+1}^{(2)} < \eta A} & = \Xi(s_1).
\end{aligned}
\]
Since (by the same consideration as in the proof of Lemma~\ref{Lem3}) $\Xi(s)$ is non-increasing function of  $s$, it follows that
\[
\Pb \brc{ M_{n+1}^{(3)} \le x | M_n ^{(3)} = t, M_{n+1}^{(3)} < \eta A} \le \Pb \brc{ M_{n+1}^{(2)} \le x | M_n ^{(2)} = t, M_{n+1}^{(2)} < \eta A}.
\]

Now let $0 \le s \le t < \eta A$. By Lemma~\ref{Lem3}, $M_{n}^{(3)}$ is stochastically monotone, so that
\begin{equation}\label{PrM2}
\Pb \brc{ M_{n+1}^{(2)} \le x | M_n ^{(2)} = s, M_{n+1}^{(2)} < \eta A} \ge \Pb \brc{ M_{n+1}^{(3)} \le x | M_n ^{(3)} = t, M_{n+1}^{(3)} < \eta A} .
\end{equation}

Construct a sample space where $M^{(2)}_0=M^{(3)}_0=0$. By \eqref{PrM2},  $M^{(3)}_1$ is stochastically larger than $M^{(2)}_1$. Hence,  one can construct the probability space so that also $M^{(3)}_1 \ge M_1^{(2)}$. Now, due to stochastic monotonicity of the transition probabilities, $M^{(3)}_2$ is stochastically larger than $M^{(2)}_2$, and one can construct the probability space so that also $M^{(3)}_2 \ge M^{(2)}_2$. Continuing this inductively, one obtains a sample space where $M^{(3)}_n \ge M_n^{(2)}$ for all $n \ge 0$. Under $\Pb_\infty$, $M^{(3)}_n$ and $M^{(2)}_n$ tend in distribution to the quasi-stationary distributions $\Qb_{\eta A}^{(3)}$ and $\Qb_{\eta A}^{(2)}$, respectively, so it follows that $\Qb_{\eta A}^{(3)} \le  \Qb_{\eta A}^{(2)}(x)$ for all $x$.

Finally, consider the process $\widetilde{M}_n^{(3)}$ governed by
\[
\Pb \brc{\widetilde{M}_{n+1}^{(3)} \le x | \widetilde{M}_n ^{(3)} = t} = \Pb (\Lambda_1 \le x/(\eta+t))
\]
and started at $\widetilde{M}_0^{(3)}\sim \Qb_{\eta A}^{(3)}$ and the process $\widetilde{M}_n^{(2)}$ governed by
\[
\Pb \brc{\widetilde{M}_{n+1}^{(2)} \le x | \widetilde{M}_n ^{(2)} = t} = \Pb (\Lambda_1 \le x/(1+t))
\]
and started at $\widetilde{M}_0^{(2)}\sim \Qb_{\eta A}^{(2)}$. In just the same way as above, we can construct a single probability space with $\widetilde{M}_n^{(3)} \ge \widetilde{M}_n^{(2)}$ for all $n \ge 0$. Therefore, the process $\widetilde{M}_n^{(3)}$ will exit above $\eta A$  no later than the process $\widetilde{M}_n^{(2)}$, and therefore the expected exit time of $\widetilde{M}_n^{(3)}$ will not exceed that of $\widetilde{M}_n^{(2)}$.
But these expectations are ARL's to false alarm if $X_i \sim f$ and $\ADD_\infty$  if $X_i \sim g$.
Furthermore, the ARL to false alarm and  the $\ADD_\infty$ of $\widetilde{M}_n^{(1)}$
(where $\widetilde{M}_n^{(1)}=\eta^{-1} \widetilde{M}_n^{(3)}$) are equal to those of  $\widetilde{M}_n^{(3)}$. Clearly, the first exit time of $\widetilde{M}_n^{(1)}$ from $[0, A)$ is nothing but the SRP stopping time $T_A^{\Qb_A}=\inf\{n\colon R_n^{\Qb_A} \ge A\}$.
Hence, it follows that both the average delay to detection $\Eb_0 T_{A}^{\Qb_{A}}$ ($=\ADD_\infty(T_A^r)$) and the ARL to false alarm $\Eb_\infty T_{A}^{\Qb_{A}}$ of
the SRP procedure are increasing functions of $A$, and the proof is complete.
\end{proof}


\end{document}